\documentclass[12pt]{article}
\usepackage{amsmath,amsthm,amssymb}
\usepackage{amssymb,latexsym}

\usepackage{eso-pic}
\usepackage{lipsum}

\newtheorem{theorem}{Theorem}
\newtheorem{conjecture}{Conjecture}
\newtheorem{lemma}{Lemma}

\begin{document}
\baselineskip=17pt

\AddToShipoutPictureBG*{
  \AtPageUpperLeft{
  \hspace{\paperwidth}
   \raisebox{-4\baselineskip}{
   \makebox[-180pt][r]{This is the author's version of the paper. The final publication has appeared in}
     \raisebox{-1.2\baselineskip}{
      \makebox[-211pt][r]{Ramanujan J., \textbf{59}, 2, (2022), 571 -- 607.}}}}}

\title{\bf A ternary diophantine inequality by primes with one of the form $\mathbf{p=x^2+y^2+1}$}

\author{\bf S. I. Dimitrov}

\date{}

\maketitle

\begin{abstract}
In this paper we solve the ternary  Piatetski-Shapiro inequality with prime numbers of a special form.
More precisely we show that, for any fixed $1<c<\frac{427}{400}$, every sufficiently large positive number $N$
and a small constant $\varepsilon>0$, the diophantine inequality
\begin{equation*}
|p_1^c+p_2^c+p_3^c-N|<\varepsilon
\end{equation*}
has a solution in prime numbers $p_1,\,p_2,\,p_3$, such that $p_1=x^2 + y^2 +1$.
For this purpose we establish a new Bombieri -- Vinogradov type result for exponential sums over primes.
\quad\\
{\bf  2020 Math.\ Subject Classification}: 11D75  $\cdot$ 11L07 $\cdot$  11L20  $\cdot$  11P32
\end{abstract}

\section{Introduction and statement of the result}
\indent

In 1960 Linnik \cite{Linnik} showed that there exist infinitely many prime numbers of the form
$p=x^2 + y^2 +1$, where $x$ and $y$ are integers.
More precisely he proved the asymptotic formula
\begin{equation*}
\sum_{p\leq X}r(p-1)=\pi\prod_{p>2}\bigg(1+\frac{\chi_4(p)}{p(p-1)}\bigg)\frac{X}{\log X}+\mathcal{O}\bigg(\frac{X(\log\log X)^5}{(\log X)^{1+\theta_0}}\bigg)\,,
\end{equation*}
where $r(k)$ is the number of solutions of the equation $k=x^2 + y^2$ in integers, $\chi_4(k)$ is the non-principal character modulo 4 and
\begin{equation}\label{theta0}
\theta_0=\frac{1}{2}-\frac{1}{4}e\log2=0.0289...
\end{equation}
In 1992 Tolev \cite{Tolev1} proved that
for any fixed $1<c<\frac{15}{14}$, for every sufficiently large positive number $N$
and a small constant $\varepsilon>0$, the diophantine inequality
\begin{equation}\label{Inequality1}
|p_1^c+p_2^c+p_3^c-N|<\varepsilon
\end{equation}
has a solution in prime numbers $p_1,\,p_2,\,p_3$.

Subsequently the result of Tolev was improved by several authors
\cite{Baker},
\cite{Baker-Weingartner}, \cite{Cai1}, \cite{Cai2}, \cite{Cai3}, \cite{Cao-Zhai}, \cite{Ku-Ne}, \cite{Kumchev}.
The best result up to now belongs to Baker \cite{Baker} with $1<c<\frac{6}{5}$.

Motivated by these results in this paper we solve inequality \eqref{Inequality1} with prime numbers
of a special type. More precisely we shall prove solvability of \eqref{Inequality1} with Linnik primes.
In order to achieve our goal we establish  a new Bombieri -- Vinogradov  type result for exponential sums over primes.

Recall that Siegel-Walfisz and Bombieri -- Vinogradov theorems are extremely important results
in analytic number theory and have various applications.

Siegel-Walfisz theorem is a refinement both of the prime number theorem and of
Dirichlet's theorem on primes in arithmetic progressions.
It states that for any fixed  $A > 0$ there exists a positive constant $c$ depending only on $A$ such that
\begin{equation*}
\sum_{p\le x\atop{p\equiv a\,( \textmd{mod}\, d)}} \log p
= \frac{x}{\varphi(d)}+\mathcal{O}\bigg(\frac{x}{e^{c\sqrt{\log x}}}\bigg)\,,
\end{equation*}
whenever $x\geq2$, $(a,d) = 1$, $d\leq (\log x)^A$ and  $\varphi (n)$ is Euler's function.

The celebrated Bombieri -- Vinogradov theorem concerns the distribution of primes in arithmetic progressions,
averaged over a range of moduli and states the following.
Let $A > 0$ be fixed. Then
\begin{equation*}
\sum\limits_{d\le \sqrt{X}/(\log X)^{A+5}}\max\limits_{y\le X}\max\limits_{(a,\,d)=1}
\Bigg|\sum_{p\le y\atop{p\equiv a\,( \textmd{mod}\, d)}} \log p-\frac{y}{\varphi(d)}\Bigg|\ll\frac{X}{\log^AX}\,.
\end{equation*}
In 2017 Tolev \cite{Tolev2} proved a Siegel-Walfisz type result for exponential sums over primes.
It states the following.
Let $\delta$, $\xi$ and $\mu$ be positive real numbers depending on $c>1$, such that
\begin{equation*}
\xi+3\delta<\frac{12}{25}\,,\quad \mu<1\,.
\end{equation*}
Let $D=X^\delta$ and $\lambda(d)$ are real numbers satisfying
\begin{equation*}
|\lambda(d)|\leq1\,,\quad\lambda(d)=0\;\;\mbox{ if }\;\;2|d\;\;\mbox{ or }\;\;\mu(d)=0\,,
\end{equation*}
where $\mu(d)$ is M\"{o}bius' function.
If
\begin{equation*}
L(t,X)=\sum\limits_{d\le D}\lambda(d)
\sum\limits_{\substack{\mu X<p\le X \\ p+2\equiv 0\,( \textmd{mod}\, d)}} e(t p^c) \log p
\end{equation*}
then for $|t|<X^{\xi-c}$ the asymptotic formula
\begin{equation*}
L(t,X)=\left(\int\limits_{\mu X}^{X}e(ty^c)\,dy\right)\sum\limits_{d\leq D}\frac{\lambda(d)}{\varphi(d)}
+\mathcal{O}\bigg(\frac{X}{\log^AX}\bigg)\,,
\end{equation*}
holds. Here $A>0$ is an arbitrary large constant.

Motivated by these investigations in this paper we establish a new Bombieri -- Vinogradov  type
result for exponential sums over primes.
More precisely we establish the following upper bound.
Let $1<c<3$, $c\neq2$, $0<\mu<1$ and $A>0$ be fixed. Then for $|t|\leq X^{\frac{1}{4}-c}$ the inequality
\begin{equation}\label{Dimrеsult}
\sum\limits_{d\le \sqrt{X}/(\log X)^{2A+10}}\max\limits_{y\le X}\max\limits_{(a,\, d)=1}\Bigg|\sum\limits_{\mu y<p\leq y\atop{p\equiv a\, ( d)}}e(t p^c)\log p
-\frac{1}{\varphi(d)}\int\limits_{\mu y}^{y}e(t x^c)\,dx\Bigg|\ll\frac{X}{\log^AX}
\end{equation}
holds.

Using \eqref{Dimrеsult} as a main weapon we are able to attack the following theorem.
\begin{theorem}\label{Theorem}
Let $1<c<\frac{427}{400}$. For every sufficiently large positive number $N$, the diophantine inequality
\begin{equation*}
|p_1^c+p_2^c+p_3^c-N|<\frac{(\log\log N)^6}{(\log N)^{\theta_0}}
\end{equation*}
has a solution in prime numbers $p_1,\,p_2,\,p_3$, such that $p_1=x^2 + y^2 +1$.
Here $\theta_0$ is defined by \eqref{theta0}.
\end{theorem}

\vspace{1mm}

In addition we have the following tasks for the future.
\begin{conjecture}  Let $\varepsilon>0$ be a small constant.
There exists $c_0>1$  such that  for any fixed $1<c<c_0$,
and every sufficiently large positive number $N$, the diophantine inequality
\begin{equation*}
|p_1^c+p_2^c+p_3^c-N|<\varepsilon
\end{equation*}
has a solution in prime numbers $p_1,\,p_2,\,p_3$, such that
$p_1=x_1^2 + y_1^2 +1$, $p_2=x_2^2 + y_2^2 +1$, $p_3=x_3^2 + y_3^2 +1$.
\end{conjecture}

\begin{conjecture}  Let $\varepsilon>0$ be a small constant.
There exists $c_0>1$  such that  for any fixed $1<c<c_0$,
and every sufficiently large positive number $N$, the diophantine inequality
\begin{equation*}
|p_1^c+p_2^c-N|<\varepsilon
\end{equation*}
has a solution in prime numbers $p_1,\,p_2$, such that
$p_1=x_1^2 + y_1^2 +1$, $p_2=x_2^2 + y_2^2 +1$.
\end{conjecture}
Conjecture 2 is analogous to the binary Goldbach problem and probably quite difficult.

\section{Notations}
\indent

Assume that $N$ is a sufficiently large positive number. The letter $p$ with or without subscript will always denote prime numbers.
The notation $m\sim M$ means that $m$ runs through the interval $(M/2, M]$. Moreover $e(t)$=exp($2\pi it$). We denote by  $(m,n)$ the greatest common divisor of $m$ and $n$.
The letter $\eta$ denotes an arbitrary small positive number, not the same in all appearances.
As usual $\varphi (n)$ is Euler's function, $\mu(n)$ is M\"{o}bius' function, $\tau(n)$ denotes the number of positive divisors of $n$ and $\Lambda(n)$ is von Mangoldt's function.
We shall use the convention that a congruence, $m\equiv n\,\pmod {d}$ will be written as $m\equiv n\,(d)$. The letter $\chi$ denotes a Dirichlet character to a given modulus.
The sums $\sum_{\chi(d)}$ and $\sum_{\chi(d)}^*$ denotes respectively summation over all characters and all primitive characters modulo $d$.
Throughout this paper unless something else is said, we suppose that $\mu$, $c$ be fixed with $0<\mu<1$ and $1<c<\frac{427}{400}$.
Denote
\begin{align}
\label{X}
&X =\left(\frac{N}{2}\right)^{\frac{1}{c}}\,;\\
\label{D}
&D=\frac{X^{\frac{1}{2}}}{(\log X)^{2A+10}}\,,\quad A>3\,;\\
\label{Delta}
&\Delta=X^{\frac{1}{4}-c}\,;\\
\label{varepsilon}
&\varepsilon=\frac{(\log\log X)^6}{(\log X)^{\theta_0}}\,;\\
\label{H}
&H=\frac{\log^2X}{\varepsilon}\,;\\
\label{SldalphaX}
&S_{l,d;J}(t)=\sum\limits_{p\in J\atop{p\equiv l\, (d)}} e(t p^c)\log p\,;\\
\label{SalphaX}
&S(t)=S_{1,1;(X/2,X]}(t)\,;\\
\label{IJalphaX}
&I_J(t)=\int\limits_Je(t y^c)\,dy\,;\\
\label{IalphaX}
&I(t)=I_{(X/2,X]}(t)\,;\\
\label{Psi}
&\Psi(y,\chi,t)=\sum\limits_{\mu y<n\leq y}\Lambda(n)\chi(n)e(t n^c)\,;\\
\label{Eytda}
&E(y,t,d,a)=\sum\limits_{ \mu y<n\leq y\atop{n\equiv a\, (d)}}\Lambda(n)e(t n^c)
-\frac{1}{\varphi(d)}\int\limits_{ \mu y}^{y}e(t x^c)\,dx\,.
\end{align}

Here and throughout this paper we denote by $J$ an arbitrary subinterval of $(X/2,X]$.

\section{Preliminary lemmas}
\indent

\begin{lemma}\label{Fourier} Let $a, \delta\in \mathbb{R}$ ,
$0 < \delta< a/4$ and $k\in \mathbb{N}$.
There exists a function $\theta(y)$ which is $k$ times continuously differentiable and
such that
\begin{align*}
&\theta(y)=1\quad\quad\quad\mbox{for }\quad\;\; |y|\leq a-\delta\,;\\
&0<\theta(y)<1\quad\,\mbox{for}\quad\quad  a-\delta <|y|< a+\delta\,;\\
&\theta(y)=0\quad\quad\quad\mbox{for}\quad\quad|y|\geq a+\delta\,.
\end{align*}
and its Fourier transform
\begin{equation*}
\Theta(x)=\int\limits_{-\infty}^{\infty}\theta(y)e(-xy)dy
\end{equation*}
satisfies the inequality
\begin{equation*}
|\Theta(x)|\leq\min\Bigg(2a,\frac{1}{\pi|x|},\frac{1}{\pi |x|}
\bigg(\frac{k}{2\pi |x|\delta}\bigg)^k\Bigg)\,.
\end{equation*}
\end{lemma}
\begin{proof}
See (\cite{Shapiro}).
\end{proof}
Throughout this paper we denote by $\theta(y)$ the function from Lemma \ref{Fourier} with parameters $\displaystyle a = \frac{9\varepsilon}{10}$,
$\displaystyle\delta=\frac{\varepsilon}{10}$, $k=[\log X]$ and by $\Theta(x)$ the Fourier transform of $\theta(y)$.
\begin{lemma}\label{SIasympt} Let $1<c<3$, $c\neq2$ and $|t|\leq\Delta$.
Then the asymptotic formula
\begin{equation*}
\sum\limits_{\mu X<p\leq X} e(t p^c)\log p=\int\limits_{\mu X}^{X}e(t y^c)\,dy
+\mathcal{O}\left(\frac{X}{e^{(\log X)^{1/5}}}\right)
\end{equation*}
holds.
\end{lemma}
\begin{proof}
See (\cite{Tolev1}, Lemma 14).
\end{proof}

\begin{lemma}\label{SumPsixchit1}
Let $\delta$, $\xi$, $\mu$ and $c$ be positive real numbers, such that
\begin{equation*}
\xi+7\delta<2\,,\quad 3\xi+6\delta<2\,,\quad  0<\mu<1\,,\quad  c>1\,.
\end{equation*}
Let $Q=X^\delta$ and $D\geq2$. Then for $X^{-c}(\log X)^D\leq|t|\leq X^{\xi-c}$ the inequality
\begin{equation*}
\sum\limits_{1<q\le Q}\frac{1}{\varphi(q)}\sideset{}{^*}\sum\limits_{\chi(q)}\max\limits_{y\le X}\big|\Psi(y,\chi,t)\big|\ll \frac{X}{(\log X)^{\frac{D}{2}-17}}
\end{equation*}
holds.
Here $\Psi(y,\chi,t)$ is denoted by \eqref{Psi}.
\end{lemma}
\begin{proof}
See (\cite{Li2022a}, Lemma 2.8, \cite{Tolev2}, Lemma 10 and \cite{Zhu}, Lemma 4.5).
\end{proof}

\begin{lemma}\label{Polya–Vinogradov}(P\'{o}lya -- Vinogradov inequality)
Suppose that $M, N$ are positive integers and $\chi$ is a non-principal character modulo
$q$. Then
\begin{equation*}
\bigg|\sum\limits_{M<n\leq M+N}\chi(n)\bigg|\leq6\sqrt{q}\log q\,.
\end{equation*}
\end{lemma}
\begin{proof}
See (\cite{Iwaniec-Kowalski}, Theorem 12.5)
\end{proof}

\begin{lemma}\label{Isest} We have
\begin{equation*}
\int\limits_{\mu X}^{X} y^{\beta-1+i\gamma}e(t y^c)\,dy
\ll\begin{cases}
\frac{X^\beta}{|t|X^c}\quad\quad\mbox{ for }\quad|\gamma|\leq\pi c\mu^c|t| X^c\,,\\
\frac{X^\beta}{\sqrt{|t|X^c}}\quad\mbox{ for }\quad\pi c\mu^c|t| X^c<|\gamma|<4\pi c|t| X^c\,,\\
\frac{X^\beta}{|\gamma|}\quad\quad\;\mbox{ for }\quad |\gamma|\geq4\pi c|t| X^c\,.
\end{cases}
\end{equation*}
\end{lemma}
\begin{proof}
See (\cite{Tolev2}, Lemma 10).
\end{proof}

\begin{lemma}\label{largesieve} (Large Sieve)
For any complex numbers $a_n$ and positive integers  $M, N, Q$ we have
\begin{equation*}
\sum\limits_{q\leq Q}\frac{q}{\varphi(q)}\sideset{}{^*}
\sum\limits_{\chi(q)}\bigg|\sum\limits_{n=M+1}^{M+N}a_n\chi(n)\bigg|^2
\ll\big(N + Q^2\big)\sum\limits_{n=M+1}^{M+N}|a_n|^2
\end{equation*}
\end{lemma}
\begin{proof}
See (\cite{Iwaniec-Kowalski}, Theorem 7.13).
\end{proof}

\begin{lemma}\label{intLintI}
For the sum denoted by \eqref{SalphaX} and the integral denoted by \eqref{IalphaX} we have
\begin{align*}
&\emph{(i)}\quad\quad\quad\;\,
\int\limits_{-\Delta}^\Delta|S(t)|^2\,dt\,\ll X^{2-c}\log^3X\,,
\quad\quad\quad\quad\quad\quad\quad\\
&\emph{(ii)}\quad\quad\quad\int\limits_{-\Delta}^\Delta|I(t)|^2\,dt\ll X^{2-c}\log X\,,\\
\quad\quad\quad\quad\quad\quad\quad
&\emph{(iii)}\quad\quad\;\,
\int\limits_{n}^{n+1}|S(t)|^2\,dt\ll X\log^3X\,.
\quad\quad\quad\quad\quad\quad\quad
\end{align*}
\end{lemma}
\begin{proof}
It follows from the arguments used in (\cite{Tolev1}, Lemma 7).
\end{proof}

\begin{lemma}\label{intSld}
For the sum denoted by \eqref{SldalphaX} uniformly for $l$ and $J$ we have
\begin{equation*}
\int\limits_{-\Delta}^\Delta|S_{l,d;J}(t)|^2\,dt\ll\frac{X^{2-c}\log^3X}{d^2}\,.
\end{equation*}
\end{lemma}
\begin{proof}
It follows by the arguments used in (\cite{Dimitrov1}, Lemma 6 (i)).
\end{proof}

\begin{lemma}\label{IestTitchmarsh}
Assume that $F(x)$, $G(x)$ are real functions defined on $[a,b]$,
$|G(x)|\leq H$ for $a\leq x\leq b$ and $G(x)/F'(x)$ is a monotonic function. Set
\begin{equation*}
I=\int\limits_{a}^{b}G(x)e(F(x))dx\,.
\end{equation*}
If $F'(x)\geq h>0$ for all $x\in[a,b]$ or if $F'(x)\leq-h<0$ for all $x\in[a,b]$ then
\begin{equation*}
|I|\ll H/h\,.
\end{equation*}
If $F''(x)\geq h>0$ for all $x\in[a,b]$ or if $F''(x)\leq-h<0$ for all $x\in[a,b]$ then
\begin{equation*}
|I|\ll H/\sqrt h\,.
\end{equation*}
\end{lemma}
\begin{proof} See (\cite{Titchmarsh}, p. 71).
\end{proof}

\begin{lemma}\label{Squareout}
For any complex numbers $a(n)$ we have
\begin{equation*}
\bigg|\sum_{a<n\le b}a(n)\bigg|^2
\leq\bigg(1+\frac{b-a}{Q}\bigg)\sum_{|q|\leq Q}\bigg(1-\frac{|q|}{Q}\bigg)
\sum_{a<n,\, n+q\leq b}a(n+q)\overline{a(n)},
\end{equation*}
where $Q\geq1$.
\end{lemma}
\begin{proof}
See (\cite{Iwaniec-Kowalski}, Lemma 8.17).
\end{proof}

\begin{lemma}\label{Heath-Brown} Let $3 < U < V < Z < X$ and suppose that $Z -\frac{1}{2}\in\mathbb{N}$,
$X\gg Z^2U$, $Z \gg U^2$, $V^3\gg X$.
Assume further that $F(n)$ is a complex valued function such that $|F(n)| \leq 1$.
Then the sum
\begin{equation*}
\sum\limits_{n\sim X}\Lambda(n)F(n)
\end{equation*}
can be decomposed into $O\Big(\log^{10}X\Big)$ sums, each of which is either of Type I
\begin{equation*}
\sum\limits_{m\sim M}a(m)\sum\limits_{l\sim L}F(ml)\,,
\end{equation*}
where
\begin{equation*}
L \gg Z\,, \quad  LM\asymp X\,, \quad |a(m)|\ll m^\eta\,,
\end{equation*}
or of Type II
\begin{equation*}
\sum\limits_{m\sim M}a(m)\sum\limits_{l\sim L}b(l)F(ml)\,,
\end{equation*}
where
\begin{equation*}
U \ll L \ll V\,, \quad  LM\asymp X\,, \quad |a(m)|\ll m^\eta\,,\quad |b(l)|\ll l^\eta\,.
\end{equation*}
\end{lemma}
\begin{proof}
See (\cite{Heath}, Lemma 3).
\end{proof}

\begin{lemma}\label{Exponentpairs}
Let $|f^{(m)}(u)|\asymp YX^{1-m}$  for $1\leq X<u<X_0\leq2X$ and $m\geq1$.\\
Then
\begin{equation*}
\bigg|\sum_{X<n\le X_0}e(f(n))\bigg|
\ll Y^\varkappa X^\lambda +Y^{-1},
\end{equation*}
where $(\varkappa, \lambda)$ is any exponent pair.
\end{lemma}
\begin{proof}
See (\cite{Graham-Kolesnik}, Ch. 3).
\end{proof}

\begin{lemma}\label{Baker-Weingartnerest} Let $\theta$, $\lambda$ be real numbers such that
\begin{equation*}
\theta(\theta-1)(\theta-2)\lambda(\lambda-1)(\theta+\lambda-2)(\theta+\lambda-3)
(\theta+2\lambda-3)(2\theta+\lambda-4)\neq0\,.
\end{equation*}
Set
\begin{equation*}
\Sigma_I=\sum\limits_{m\sim M}a(m)\sum\limits_{l\in I_m}e(Bm^\lambda l^\theta)\,,
\end{equation*}
where
\begin{equation*}
B>0\,, \quad M\geq1\,, \quad L\geq1\,, \quad |a(m)|\leq1\,, \quad I_m\subset(L/2,L]\,.
\end{equation*}
Let
\begin{equation*}
F=BM^\lambda L^\theta\,.
\end{equation*}
Then
\begin{align*}
\Sigma_I\ll\Big(F^{\frac{3}{14}}M^{\frac{41}{56}}L^{\frac{29}{56}}&
+F^{\frac{1}{5}}M^{\frac{3}{4}}L^{\frac{11}{20}}
+F^{\frac{1}{8}}M^{\frac{13}{16}}L^{\frac{11}{16}}\\
&+M^{\frac{3}{4}}L+ML^{\frac{3}{4}}+F^{-1}ML\Big)(ML)^\eta\,.
\end{align*}
\end{lemma}
\begin{proof}
See (\cite{Baker-Weingartner}, Theorem 2).
\end{proof}

\begin{lemma}\label{Sargos-Wuest} Let  $\alpha$, $\beta$ be real numbers such that
\begin{equation*}
\alpha\beta(\alpha-1)(\beta-1)(\alpha-2)(\beta-2)\neq0\,.
\end{equation*}
Set
\begin{equation*}
\Sigma_{II}=\sum\limits_{m\sim M}a(m)
\sum\limits_{l\sim L}b(l)e\left(F\frac{m^\alpha l^\beta}{M^\alpha L^\beta}\right)\,,
\end{equation*}
where
\begin{equation*}
F>0\,, \quad M\geq1\,, \quad L\geq1\,, \quad |a(m)|\leq1\,, \quad|b(l)|\leq1\,.
\end{equation*}
Then
\begin{align*}
\Sigma_{II}(FML)^{-\eta}&\ll(F^4M^{31}L^{34})^{\frac{1}{42}}+(F^6M^{53}L^{51})^{\frac{1}{66}}+(F^6M^{46}L^{41})^{\frac{1}{56}}\\
&+(F^2M^{38}L^{29})^{\frac{1}{40}}+(F^3M^{43}L^{32})^{\frac{1}{46}}+(FM^9L^6)^{\frac{1}{10}}\\
&+(F^2M^7L^6)^{\frac{1}{10}}+(FM^6L^6)^{\frac{1}{8}}+M^{\frac{1}{2}}L\\
&+ML^{\frac{1}{2}}+F^{-\frac{1}{2}}ML\,.
\end{align*}
\end{lemma}
\begin{proof}
See (\cite{Sargos-Wu}, Theorem  9).
\end{proof}
The next two lemmas are due to C. Hooley.
\begin{lemma}\label{Hooley1}
For any constant $\omega>0$ we have
\begin{equation*}
\sum\limits_{p\leq X}
\bigg|\sum\limits_{d|p-1\atop{\sqrt{X}(\log X)^{-\omega}<d<\sqrt{X}(\log X)^{\omega}}}
\chi_4(d)\bigg|^2\ll \frac{X(\log\log X)^7}{\log X}\,,
\end{equation*}
where the constant in Vinogradov's symbol depends on $\omega>0$.
\end{lemma}

\begin{lemma}\label{Hooley2} Suppose that $\omega>0$ is a constant
and let $\mathcal{F}_\omega(X)$ be the number of primes $p\leq X$
such that $p-1$ has a divisor in the interval $\big(\sqrt{X}(\log X)^{-\omega}, \sqrt{X}(\log X)^\omega\big)$.
Then
\begin{equation*}
\mathcal{F}_\omega(X)\ll\frac{X(\log\log X)^3}{(\log X)^{1+2\theta_0}}\,,
\end{equation*}
where $\theta_0$ is defined by \eqref{theta0} and the constant in
Vinogradov's symbol depends only on $\omega>0$.
\end{lemma}
The proofs of very similar results are available in (\cite{Hooley}, Ch.5).

\begin{lemma}\label{IIIest} We have
\begin{equation*}
\int\limits_{-\infty}^{\infty} I^3(t)\Theta(t)e(-Nt)\,dt\gg \varepsilon X^{3-c}\,.
\end{equation*}
\end{lemma}
\begin{proof}
See (\cite{Tolev1}, Lemma 6).
\end{proof}

\section{Outline of the proof}
\indent

Consider the sum
\begin{equation}\label{Gamma}
\Gamma(X)=
\sum\limits_{X/2<p_1,p_2,p_3\leq X\atop{|p_1^c+p_2^c+p_3^c-N|<\varepsilon}}
r(p_1-1)\log p_1\log p_2\log p_3\,.
\end{equation}
Obviously
\begin{equation}\label{GammaGamma0}
\Gamma(X)\geq\Gamma_0(X)\,,
\end{equation}
where
\begin{equation}\label{Gamma0}
\Gamma_0(X)=\sum\limits_{X/2<p_1,p_2,p_3\leq X}r(p_1-1)
\theta\big(p_1^c+p_2^c+p_3^c-N\big)\log p_1 \log p_2\log p_3\,.
\end{equation}
From \eqref{Gamma0} and well-known identity
\begin{equation*}
r(n)=4\sum_{d|n}\chi_4(d)
\end{equation*}
we obtain
\begin{equation} \label{Gamma0decomp}
\Gamma_0(X)=4\big(\Gamma_1(X)+\Gamma_2(X)+\Gamma_3(X)\big),
\end{equation}
where
\begin{align}
\label{Gamma1}
&\Gamma_1(X)=\sum\limits_{X/2<p_1,p_2,p_3\leq X}
\left(\sum\limits_{d|p_1-1\atop{d\leq D}}\chi_4(d)\right)
\theta\big(p_1^c+p_2^c+p_3^c-N\big)\log p_1\log p_2\log p_3\,,\\
\label{Gamma2}
&\Gamma_2(X)=\sum\limits_{X/2<p_1,p_2,p_3\leq X}
\left(\sum\limits_{d|p_1-1\atop{D<d<X/D}}\chi_4(d)\right)
\theta\big(p_1^c+p_2^c+p_3^c-N\big)\log p_1\log p_2\log p_3\,,\\
\label{Gamma3}
&\Gamma_3(X)=\sum\limits_{X/2<p_1,p_2,p_3\leq X}
\left(\sum\limits_{d|p_1-1\atop{d\geq X/D}}\chi_4(d)\right)
\theta\big(p_1^c+p_2^c+p_3^c-N\big)\log p_1\log p_2\log p_3\,.
\end{align}
In order to estimate $\Gamma_1(X)$ and $\Gamma_3(X)$ we need to consider
the sum
\begin{equation} \label{Ild}
I_{l,d;J}(X)=\sum\limits_{X/2<p_2,p_3\leq X\atop{p_1\equiv l\,(d)
\atop{p_1\in J}}}\theta\big(p_1^c+p_2^c+p_3^c-N\big)
\log p_1\log p_2\log p_3\,,
\end{equation}
where $l$ and $d$ are coprime natural numbers, and $J\subset(X/2,X]$ is an interval.
If $J=(X/2,X]$ then we write for simplicity $I_{l,d}(X)$.

Using the inverse Fourier transform for the function $\theta(y)$ we deduce
\begin{align*}
I_{l,d;J}(X)&=\sum\limits_{X/2<p_2,p_3\leq X\atop{p_1\equiv l\,(d)\atop{p_1\in J}}}
\log p_1\log p_2\log p_3
\int\limits_{-\infty}^{\infty}\Theta(t)e\big((p_1^c+p_2^c+p_3^c-N)t\big)\,dt\\
&=\int\limits_{-\infty}^{\infty}
\Theta(t)S^2(t)S_{l,d;J}(t)e(-Nt)\,dt\,.
\end{align*}
We decompose $I_{l,d;J}(X)$ over major, minor and trivial arcs as follows
\begin{equation}\label{Ilddecomp}
I_{l,d;J}(X)=I^{(1)}_{l,d;J}(X)+I^{(2)}_{l,d;J}(X)+I^{(3)}_{l,d;J}(X)\,,
\end{equation}
where
\begin{align}
\label{Ild1}
&I^{(1)}_{l,d;J}(X)=\int\limits_{-\Delta}^{\Delta}\Theta(t)S^2(t)S_{l,d;J}(t)e(-Nt)\,dt\,,\\
\label{Ild2}
&I^{(2)}_{l,d;J}(X)=\int\limits_{\Delta\leq|t|\leq H}\Theta(t)S^2(t)S_{l,d;J}(t)e(-Nt)\,dt\,,\\
\label{Ild3}
&I^{(3)}_{l,d;J}(X)=\int\limits_{|t|>H}\Theta(t)S^2(t)S_{l,d;J}(t)e(-N t)\,dt\,.
\end{align}
We shall estimate $I^{(1)}_{l,d;J}(X)$, $I^{(3)}_{l,d;J}(X)$,
$\Gamma_3(X),\,\Gamma_2(X)$ and $\Gamma_1(X)$, respectively,
in the sections \ref{SectionIld1}, \ref{SectionIld3},
\ref{SectionGamma3}, \ref{SectionGamma2} and \ref{SectionGamma1}.
In section \ref{Sectionfinal} we shall finalize the proof of Theorem \ref{Theorem}.

\section{Asymptotic formula for $\mathbf{I^{(1)}_{l,d;J}(X)}$}\label{SectionIld1}
\indent

A key point in the proof of our theorem is the following  Bombieri -- Vinogradov
type result for exponential sums over primes.
\begin{lemma}\label{Bomb-Vin-Dim} Let $1<c<3$, $c\neq2$, $0<\mu<1$, $|t|\leq X^{\frac{1}{4}-c}$, $A>0$ and $X>2$. Then 
\begin{equation*}
\sum\limits_{d\le \sqrt{X}/(\log X)^{2A+10}}\max\limits_{y\le X}\max\limits_{(a,\, d)=1}\Bigg|\sum\limits_{\mu y<p\leq y\atop{p\equiv a\, ( d)}}e(t p^c)\log p
-\frac{1}{\varphi(d)}\int\limits_{\mu y}^{y}e(t x^c)\,dx\Bigg|\ll\frac{X}{\log^AX}\,.
\end{equation*}
\end{lemma}
\begin{proof}
First we shall prove that if $0<\mu<1$, $c>1$, $B>0$, $C>0$ and $|t|\leq X^{\frac{2}{3}-c-\delta}$ for a sufficiently small $\delta>0$, then
\begin{equation}\label{0SumPsixchit2}
\sum\limits_{1<q\le \log^CX}\frac{1}{\varphi(q)}\sideset{}{^*}\sum\limits_{\chi(q)}\Bigg|\sum\limits_{\mu X<n\leq X}\Lambda(n)\chi(n)e(t n^c)\Bigg|\ll \frac{X}{\log^BX}\,.
\end{equation}
We consider two cases.

\textbf{Case 1} 

\begin{equation}\label{0tXcB}
|t|\leq X^{-c}(\log X)^{2B+34}\,.
\end{equation}
From (\cite{Davenport}, p. 132) we have that if $q\leq\log^AX$, then
\begin{equation}\label{Davenport}
\sum\limits_{n\leq X}\Lambda(n)\chi(n)\ll\frac{X}{e^{c(\log X)^{1/2}}}
\end{equation}
for any nonprincipal character $\chi$ modulo $q$. Now \eqref{0tXcB}, \eqref{Davenport} and Abel’s summation formula lead to
\begin{align*}
&\sum\limits_{1<q\le \log^CX}\frac{1}{\varphi(q)}\sideset{}{^*}\sum\limits_{\chi(q)}\Bigg|\sum\limits_{\mu X<n\leq X}\Lambda(n)\chi(n)e(t n^c)\Bigg|\nonumber\\
&=\sum\limits_{1<q\le \log^CX}\frac{1}{\varphi(q)}\sideset{}{^*}\sum\limits_{\chi(q)}\Bigg|e(tX^c)\sum\limits_{\mu X<n\leq X}\Lambda(n)\chi(n)\nonumber\\
&\hspace{40mm}-\int\limits_{\mu X}^X\left(\sum\limits_{\mu X<n\leq y}\Lambda(n)\chi(n)\right)\frac{d}{dy}e(t y^c)\,dy\Bigg|\nonumber\\
&\ll\big(1+|t|X^c\big)\frac{X}{e^{(\log X)^{1/3}}}\ll\frac{X}{e^{(\log X)^{1/4}}}\,.
\end{align*}

\textbf{Case 2} 

\begin{equation*}
X^{-c}(\log X)^{2B+34}\leq|t|\leq X^{\frac{2}{3}-c-\delta}
\end{equation*}
This case follows from Lemma \ref{SumPsixchit1}. This proves the inequality \eqref{0SumPsixchit2}. We will use the formula
\begin{equation}\label{0Lambdalog}
\sum\limits_{\mu y<p\leq y\atop{p\equiv a\, (d)}}e(t p^c)\log p=\sum\limits_{\mu y<n\leq y\atop{n\equiv a\, (d)}}\Lambda(n)e(t n^c)+\mathcal{O}\left(\frac{y^{\frac{1}{2}+\varepsilon}}{d}\right)
\end{equation}
for $d\le y^{\frac{1}{2}}$. Define
\begin{equation}\label{0deltachi}
\delta(\chi)=\begin{cases}1 \quad\mbox{ if $\chi$ is principal },\\
0\quad\mbox{ otherwise }.
\end{cases}
\end{equation}
By the orthogonality of characters we have
\begin{align*}
&\sum\limits_{\mu y<n\leq y\atop{n\equiv a\, (d)}}\Lambda(n)e(t n^c)
-\frac{1}{\varphi(d)}\int\limits_{\mu y}^{y}e(t x^c)\,dx\\
&=\sum\limits_{\mu y<n\leq y}\Lambda(n)e(t n^c)
\frac{1}{\varphi(d)}\sum\limits_{\chi(d)}\chi(n)\overline{\chi}(a)
-\frac{1}{\varphi(d)}\int\limits_{\mu y}^{y}e(t x^c)\,dx\\
&=\frac{1}{\varphi(d)}\sum\limits_{\chi(d)}\left(\overline{\chi}(a)
\sum\limits_{\mu y<n\leq y}\Lambda(n)\chi(n)e(t n^c)
-\delta(\chi)\int\limits_{\mu y}^{y}e(t x^c)\,dx\right)
\end{align*}
and therefore
\begin{align}\label{0maxLambdaInt}
&\max\limits_{(a,\, d)=1}\bigg|\sum\limits_{\mu y<n\leq y\atop{n\equiv a\, (d)}}\Lambda(n)e(t n^c)
-\frac{1}{\varphi(d)}\int\limits_{\mu y}^{y}e(t x^c)\,dx\bigg|\nonumber\\
&\leq\frac{1}{\varphi(d)}\sum\limits_{\chi(d)}\bigg|\Psi(y,\chi,t)
-\delta(\chi)\int\limits_{\mu y}^{y}e(t x^c)\,dx\bigg|\,,
\end{align}
where $\Psi(y,\chi,t)$ is defined by \eqref{Psi}.
Denote
\begin{equation}\label{0Sigma}
\Sigma=\sum\limits_{d\le \sqrt{X}/(\log X)^{2A+10}}\max\limits_{y\le X}\max\limits_{(a,\, d)=1}\big|E(y,t,d,a)\big|\,.
\end{equation}
From \eqref{Eytda}, \eqref{0deltachi}, \eqref{0maxLambdaInt} and \eqref{0Sigma} we  obtain
\begin{equation}\label{0Sigmadecomp}
\Sigma\leq\Sigma'+\Sigma''\,,
\end{equation}
where
\begin{align}
\label{0Sigma'}
&\Sigma'=\sum\limits_{d\le \sqrt{X}/(\log X)^{2A+10}}\frac{1}{\varphi(d)}\max\limits_{y\le X}\bigg|\sum\limits_{\mu y<n\leq y}\Lambda(n)e(t n^c)-\int\limits_{\mu y}^{y}e(t x^c)\,dx+\mathcal{O}\Big(\log^2y\Big)\bigg|\,,\\
\label{0Sigma''}
&\Sigma''=\sum\limits_{d\le \sqrt{X}/(\log X)^{2A+10}}\frac{1}{\varphi(d)}\sum\limits_{\chi(d)\atop{\chi\neq\chi_0}}\max\limits_{y\le X}\big|\Psi(y,\chi,t)\big|\,.
\end{align}
By \eqref{0Lambdalog}, \eqref{0Sigma'} and Lemma \ref{SIasympt} we find
\begin{equation}\label{0Sigma'est}
\Sigma'\ll\frac{X}{e^{(\log X)^{1/5}}}\sum\limits_{d\le \sqrt{X}/(\log X)^{2A+10}}\frac{1}{\varphi(d)}\ll\frac{X}{\log^AX}\,.
\end{equation}
Next we consider $\Sigma''$. Moving to primitive characters from \eqref{0Sigma''} we deduce
\begin{align}\label{0Sigma''primitive}
\Sigma''&\ll\sum\limits_{d\le \sqrt{X}/(\log X)^{2A+10}}\frac{1}{\varphi(d)}
\sum\limits_{r|d\atop{r>1}}\sideset{}{^*}\sum\limits_{\chi(r)}\max\limits_{y\le X}\big|\Psi(y,\chi,t)\big|+\frac{\sqrt{X}}{(\log X)^{2A+8}}\nonumber\\
&\ll\sum\limits_{r\le \sqrt{X}/(\log X)^{2A+10}}\left(\sum\limits_{d\le \sqrt{X}/(\log X)^{2A+10}\atop{d\equiv 0\, ( r)}}\frac{1}{\varphi(d)}\right)
\sideset{}{^*}\sum\limits_{\chi(r)}\max\limits_{y\le X}\big|\Psi(y,\chi,t)\big|+\frac{\sqrt{X}}{(\log X)^{2A+8}}\nonumber\\
&\ll(\log X)\sum\limits_{1<r\le \sqrt{X}/(\log X)^{2A+10}}\frac{1}{\varphi(r)}\sideset{}{^*}\sum\limits_{\chi(r)}\max\limits_{y\le X}\big|\Psi(y,\chi,t)\big|+\frac{\sqrt{X}}{(\log X)^{2A+8}}\nonumber\\
&=(\Omega_1+\Omega_2)\log X+\frac{\sqrt{X}}{(\log X)^{2A+8}}\,,
\end{align}
where
\begin{align}
\label{0Omega1}
&\Omega_1=\sum\limits_{r\le R_0}\frac{1}{\varphi(r)}\sideset{}{^*}\sum\limits_{\chi(r)}\max\limits_{y\le X}\big|\Psi(y,\chi,t)\big|\,,\\
\label{0Omega2}
&\Omega_2=\sum\limits_{R_0<r\le R}\frac{1}{\varphi(r)}\sideset{}{^*}\sum\limits_{\chi(r)}\max\limits_{y\le X}\big|\Psi(y,\chi,t)\big|\,,\\
\label{0R0R}
&R_0=(\log X)^{A+5}\,, \quad R=\frac{\sqrt{X}}{(\log X)^{2A+10}}\,.
\end{align}
Taking into account \eqref{Psi}, \eqref{0SumPsixchit2} with $B=A+1$, \eqref{0Omega1} and \eqref{0R0R} we obtain
\begin{equation}\label{0Omega1est}
\Omega_1\ll\frac{X}{(\log X)^{A+1}}\,.
\end{equation}
Next we consider $\Omega_2$. Let $1<u\leq \mu y$ be a parameter that we will choose later. Using \eqref{Psi} and Vaughan's identity (see \cite{Vaughan}) we get
\begin{equation}\label{0Psidecomp}
\Psi(y,\chi,t)=U_1(y,\chi,t)-U_2(y,\chi,t)-U_3(y,\chi,t)-U_4(y,\chi,t)\,,
\end{equation}
where
\begin{align}
\label{0U1}
&U_1(y,\chi,t)=\sum_{d\le u}\mu(d)\sum_{\mu y<dl\le y}\chi(dl)e(td^cl^c)\log l\,,\\
\label{0U2}
&U_2(y,\chi,t)=\sum_{d\le u}c(d)\sum_{\mu y<dl\le y}\chi(dl)e(td^cl^c)\,,\\
\label{0U3}
&U_3(y,\chi,t)=\sum_{u<d\le u^2}c(d)\sum_{\mu y<dl\le y}\chi(dl)e(td^cl^c)\,,\\
\label{0U4}
&U_4(y,\chi,t)= \mathop{\sum\sum}_{\substack{d>u,\,l>u\\\mu y<dl\le y }}a(d)\Lambda(l) \chi(dl)e(td^cl^c)\,,
\end{align}
and where
\begin{equation}\label{0cdad}
|c(d)|\leq\log d\,,\quad  | a(d)|\leq\tau(d)\,.
\end{equation}
Now \eqref{0Omega2}, \eqref{0Psidecomp} -- \eqref{0U4} give us
\begin{equation}\label{0Omega2est1}
\Omega_2\ll \Omega^{(1)}_2+\Omega^{(2)}_2+\Omega^{(3)}_2+\Omega^{(4)}_2\,,
\end{equation}
where
\begin{equation}\label{0Omegaj}
\Omega^{(j)}_2=\sum\limits_{R_0<r\le R}\frac{1}{\varphi(r)}\sideset{}{^*}\sum\limits_{\chi(r)}\max\limits_{y\le X}\big|U_j(y,\chi,t)\big|\,, \quad j=1,\,2,\,3,\,4\,.
\end{equation}

\newpage

\textbf{Estimation of $\mathbf{\Omega^{(1)}_2}$ and $\mathbf{\Omega^{(2)}_2}$}

From \eqref{0R0R}, \eqref{0U1}, \eqref{0Omegaj}, Abel's summation formula and Lemma \ref{Polya–Vinogradov} it follows that

\begin{align}\label{0Omega21est}
\Omega^{(1)}_2&\ll\sum\limits_{R_0<r\le R}\frac{1}{\varphi(r)}\sideset{}{^*}\sum\limits_{\chi(r)}
\max\limits_{y\le X}\sum_{d\le u}\bigg|\sum_{\mu y/d<l\le y/d}\chi(dl)e(td^cl^c)\log l\bigg|\nonumber\\
&\ll X^{\frac{1}{4}}(\log X)\sum\limits_{R_0<r\le R}\frac{1}{\varphi(r)}\sideset{}{^*}\sum\limits_{\chi(r)}
\sum_{d\le u}\max\limits_{y\le X}\max\limits_{\mu y/d<x\le y/d}\bigg|\sum_{\mu y/d<l\le x}\chi(l)\bigg|\nonumber\\
&\ll X^{\frac{1}{4}}(\log X)u\sum\limits_{R_0<r\le R}r^{\frac{1}{2}}\log r\nonumber\\
&\ll X^{\frac{1}{4}} u R^{\frac{3}{2}}\log^2X\,.
\end{align}
Working in a similar way we deduce
\begin{equation}\label{0Omega22est}
\Omega^{(2)}_2\ll X^{\frac{1}{4}} u R^{\frac{3}{2}}\log X\,.
\end{equation}

\textbf{Estimation of $\mathbf{\Omega^{(3)}_2}$ and $\mathbf{\Omega^{(4)}_2}$}

We split the range of $l$ of the exponential sum \eqref{0U4} into dyadic subintervals of the form $L<l\leq 2L$, where $u<L\leq y/2u$.
Further we use \eqref{0cdad}, Abel's summation formula, Perron's formula (see \cite{Tenenbaum}, Chapter II.2, Theorem 1) with parameters
\begin{equation}\label{0kappaT}
\varkappa=\frac{1}{\log X}\,, \quad T=X^2
\end{equation}
and partial integration to find
\begin{align}\label{0U4est1}
U_4(y,\chi,t)&\ll (\log X)\bigg|\sum_{l\sim L}\sum_{\mu y/l<d\leq y/ l}\Lambda(l)a(d)\chi(dl)e(td^cl^c)\bigg|\nonumber\\
&=(\log X)\Bigg|e(ty^c)\sum_{l\sim L}\sum_{\mu y/l<d\leq y/ l}\Lambda(l)a(d)\chi(dl)\nonumber\\
&-\int\limits_{\mu y/l}^{y/l}\left(\sum_{l\sim L}\sum_{\mu y/l<d\leq x}\Lambda(l)a(d)\chi(dl)\right)\,de(tx^cl^c)\Bigg|\nonumber\\
&=(\log X)\big|e(ty^c)\mathfrak{X}_1-\mathfrak{X}_2\big|\,,
\end{align}
where
\begin{align}\label{0mathfrakX1}
\mathfrak{X}_1&=\frac{1}{2\pi i}\int\limits_{1-\varkappa-iT}^{1+\varkappa+iT}\sum_{l\sim L}\sum_{\mu y/2L<d\leq X/L}\frac{\Lambda(l)a(d)\chi(dl)}{(dl)^s}\frac{y^s}{s}\,ds\nonumber\\
&\hspace{50mm}+\mathcal{O}\Bigg(\sum_{l\sim L}\sum\limits_{\mu y/2L<d\leq y/L}\frac{y^{1+\varkappa}\Lambda(l)\tau(d)}{(dl)^{1+\varkappa}\big(1+T\left|\log\frac{y}{dl}\right|\big)}\Bigg)\,,
\end{align}

\begin{align}\label{0mathfrakX2}
\mathfrak{X}_2&=\int\limits_{\mu y/l}^{y/l}\left(\frac{1}{2\pi i}\int\limits_{1-\varkappa-iT}^{1+\varkappa+iT}\sum_{l\sim L}\sum_{\mu y/2L<d\leq X/L}\frac{\Lambda(l)\chi(dl)a(d)}{(dl)^s}\frac{x^s}{s}\,ds\right.\nonumber\\
&\left.\hspace{30mm}+\mathcal{O}\Bigg(\sum_{l\sim L}\sum\limits_{\mu y/2L<d\leq y/L}\frac{x^{1+\varkappa}\Lambda(l)\tau(d)}{(dl)^{1+\varkappa}\left(1+T\left|\log\frac{x}{dl}\right|\right)}\Bigg)\right)\,de(tx^cl^c)\,.
\end{align}
If we assume, as we may, that $\|y\|=\frac{1}{2}$, we have $|\log\frac{y}{dl}|\geq\frac{1}{y}$.
Now \eqref{0kappaT}, \eqref{0mathfrakX1}, \eqref{0mathfrakX2}, partial integration and the well known inequalities 
\begin{equation}\label{Lambdatauest}
\sum\limits_{n\le X} \Lambda(n)\ll X\,, \quad \sum\limits_{n\le X} \tau(n)\ll X\log X    
\end{equation}
imply
\begin{align}
\label{0mathfrakX1est1}
&\mathfrak{X}_1=\frac{1}{2\pi i}I_1+\mathcal{O}\big(\log X\big)\,,\\
\label{0mathfrakX2est1}
&\mathfrak{X}_2=\frac{1}{2\pi i}\big(e(t\mu^cy^c)I_2-e(ty^c)I_3+I_4\big)+\mathcal{O}\Big(X^{\frac{1}{4}}u^{-1}\log X\Big)\,,
\end{align}
where
\begin{align}
\label{0I1}
&I_1=\int\limits_{1-\varkappa-iT}^{1+\varkappa+iT}\sum_{l\sim L}\sum_{\mu y/2L<d\leq X/L}\frac{\Lambda(l)a(d)\chi(dl)}{(dl)^s}\frac{y^s}{s}\,ds\,,\\
\label{0I2}
&I_2=\int\limits_{1-\varkappa-iT}^{1+\varkappa+iT}\sum_{l\sim L}\sum_{\mu y/2L<d\leq X/L}\frac{\Lambda(l)a(d)\chi(dl)}{(dl^2)^s}\frac{\mu^sy^s}{s}\,ds\,,\\
\label{0I3}
&I_3=\int\limits_{1-\varkappa-iT}^{1+\varkappa+iT}\sum_{l\sim L}\sum_{\mu y/2L<d\leq X/L}\frac{\Lambda(l)a(d)\chi(dl)}{(dl^2)^s}\frac{y^s}{s}\,ds\,,\\
\label{0I4}
&I_4=\int\limits_{1-\varkappa-iT}^{1+\varkappa+iT}\sum_{l\sim L}\sum_{\mu y/2L<d\leq X/L}\frac{\Lambda(l)a(d)\chi(dl)}{(dl^2)^s}\left(\int\limits_{\mu y}^{y}x^{s-1}e(tx^c)\,dx\right)\,ds\,.
\end{align}
Put $s=\beta+i\gamma$. Using \eqref{0kappaT}, \eqref{Lambdatauest}, \eqref{0I1} -- \eqref{0I4}, Lemma \ref{Isest} and Cauchy's integral theorem for the rectangle      
\begin{equation*}
\{\varkappa\leq\beta\leq1+\varkappa\,,\; -T\leq\gamma\leq T\}    
\end{equation*}
we derive 
\begin{align}
\label{0I1est1}
&I_1\ll \left|\sum_{l\sim L}\sum_{\mu y/2L<d\leq X/L}\frac{\Lambda(l)a(d)\chi(dl)}{(dl)^{\varkappa+i\gamma_1}}\right|\log X+\log X\,,\\
\label{0I23est1}
&I_2, I_3\ll \left|\sum_{l\sim L}\sum_{\mu y/2L<d\leq X/L}\frac{\Lambda(l)a(d)\chi(dl)}{(dl^2)^{\varkappa+i\gamma_2}}\right|\log X+\log X\,,\\
\label{0I4est1}
&I_4\ll |I_5|+\log X\,,
\end{align}
for some $|\gamma_1|, |\gamma_2|\leq T$, where
\begin{equation}\label{0I5}
I_5=\int\limits_{-T}^{T}\sum_{l\sim L}\sum_{\mu y/2L<d\leq X/L}\frac{\Lambda(l)a(d)\chi(dl)}{(dl^2)^{\varkappa+i\gamma}}\left(\int\limits_{\mu y}^{y}x^{\varkappa-1+i\gamma}e(tx^c)\,dx\right)\,d\gamma\,.
\end{equation}
Firstly, consider the case 
\begin{equation}\label{0case1}
4\pi c|t|y^c\leq\log y\,.
\end{equation}
By \eqref{0kappaT}, \eqref{0I5}, \eqref{0case1} and Lemma \ref{Isest} we obtain
\begin{align}\label{0I5est1}
I_5&\ll\left|\sum_{l\sim L}\sum_{\mu y/2L<d\leq X/L}\frac{\Lambda(l)a(d)\chi(dl)}{(dl^2)^{\varkappa+i\gamma_3}}\right|
\left(\int\limits_{0}^{\log y}y^\varkappa\,d\gamma+\int\limits_{\log y}^{T}\frac{y^\varkappa}{\gamma}\,d\gamma\right)\nonumber\\
&\ll\left|\sum_{l\sim L}\sum_{\mu y/2L<d\leq X/L}\frac{\Lambda(l)a(d)\chi(dl)}{(dl^2)^{\varkappa+i\gamma_3}}\right|\log X\,,
\end{align}
for some $|\gamma_3|\leq T$. Consider now the case
\begin{equation}\label{0case2}
4\pi c|t|y^c>\log y\,.
\end{equation}
From \eqref{0kappaT}, \eqref{0I5}, \eqref{0case2} and Lemma \ref{Isest} we deduce
\begin{align}\label{0I5est2}
I_5&\ll\left|\sum_{l\sim L}\sum_{\mu y/2L<d\leq X/L}\frac{\Lambda(l)a(d)\chi(dl)}{(dl^2)^{\varkappa+i\gamma_4}}\right|
\left(\int\limits_{0}^{\pi c\mu^c|t| y^c}\frac{y^\varkappa}{|t| y^c}\,d\gamma+\int\limits_{4\pi c|t|y^c}^{T}\frac{y^\varkappa}{\gamma}\,d\gamma\right)+|I_6|\nonumber\\
&\ll\left|\sum_{l\sim L}\sum_{\mu y/2L<d\leq X/L}\frac{\Lambda(l)a(d)\chi(dl)}{(dl^2)^{\varkappa+i\gamma_4}}\right|\log X+|I_6|\,,
\end{align}
for some $|\gamma_4|\leq T$, where
\begin{equation}\label{0I6}
I_6=\int\limits_{\pi c\mu^c|t| y^c<|\gamma|<4\pi c|t| y^c}\sum_{l\sim L}\sum_{\mu y/2L<d\leq X/L}\frac{\Lambda(l)a(d)\chi(dl)}{(dl^2)^{\varkappa+i\gamma}}
\left(\int\limits_{\mu y}^{y}x^{\varkappa-1+i\gamma}e(tx^c)\,dx\right)\,d\gamma\,.
\end{equation}
With a change of variables we can write the integral with respect to $x$ in the form
\begin{equation}\label{0Ibetagammatu}
\int\limits_{\mu y}^{y}x^{\varkappa-1+i\gamma}e(tx^c)\,dx=\frac{1}c{}\int\limits_{(\mu y)^c}^{y^c}\varphi(u)e\big(F(u)\big)\,du\,,
\end{equation}
where
\begin{equation}\label{0varphiF}
F(u)=\frac{\gamma}{2\pi c}\log u+tu\,, \quad \varphi(u)=u^{\frac{1}{c}\varkappa-1}\,.
\end{equation}
We have
\begin{equation}\label{0F'}
F'(u)=\frac{\gamma}{2\pi cu}+t\,, \quad F''(u)=-\frac{\gamma}{2\pi cu^2}\,, \quad F'''(u)=\frac{\gamma}{\pi cu^3}\,,
\end{equation}
\begin{equation}\label{0varphi'}
\varphi'(u)=\left(\frac{\varkappa}{c}-1\right)u^{\frac{1}{c}\varkappa-2}\,, \quad \varphi''(u)=\left(\frac{\varkappa}{c}-1\right)\left(\frac{\varkappa}{c}-2\right)u^{\frac{1}{c}\varkappa-3}\,.
\end{equation}
Bearing in mind that $|\gamma|\asymp|t|y^c$, $u\asymp y^c$ by \eqref{0varphiF} -- \eqref{0varphi'} we get
\begin{equation}\label{0F'est}
|t|\ll F'(u)\ll |t|\,, \quad A^{-1}\ll F''(u)\ll A^{-1}\,, \quad F'''(u)\ll A^{-1}U^{-1}\,,
\end{equation}
\begin{equation}\label{0varphi'est}
\varphi(u)\ll H\,, \quad \varphi'(u)\ll HU^{-1}\,, \quad \varphi''(u)\ll HU^{-2}\,,
\end{equation}
where
\begin{equation}\label{0HUA}
H=y^{\varkappa-c}\,, \quad U=y^c\,, \quad A=y^c|t|^{-1}\,.
\end{equation}
Put
\begin{equation}\label{0uo}
u_0=-\frac{\gamma}{2\pi ct}\,.
\end{equation}
From \eqref{0F'} it follows that
\begin{equation}\label{0F'u0}
F'(u_0)=0\,. 
\end{equation}
Put
\begin{equation}\label{0ab}
a=\pi c\mu^c|t| y^c|\,, \quad b=4\pi c|t| y^c)|\,.
\end{equation}
Now \eqref{0kappaT}, \eqref{0case2}, \eqref{0I6} --  \eqref{0F'}, \eqref{0F'est} -- \eqref{0ab}, Abel's summation formula and (\cite{Karat}, Ch. 1, \S 3, Lemma 2) yield
\begin{align*}
I_6&=\int\limits_a^b\sum_{l\sim L}\sum_{\mu y/2L<d\leq X/L}\frac{\Lambda(l)a(d)\chi(dl)}{(dl^2)^{\varkappa+i\gamma}}
\Bigg[\frac{1+i}{c\sqrt{2}}\cdot\frac{\varphi(u_0)e\big(F(u_0)\big)}{\sqrt{F''(u_0)}}\nonumber\\
&+\mathcal{O}\big(HU^{-1}A\big)+\mathcal{O}\Bigg(H\min\bigg(\frac{1}{|F'(a)|}, \sqrt{A} \bigg)\Bigg)+\mathcal{O}\Bigg(H\min\bigg(\frac{1}{|F'(b)|}, \sqrt{A} \bigg)\Bigg)\Bigg]d\gamma\nonumber\\
&\ll\left|\sum_{l\sim L}\sum_{\mu y/2L<d\leq X/L}\frac{\Lambda(l)a(d)\chi(dl)}{(dl^2)^{\varkappa+i\gamma_5}\log(dl)}\right|
+\left|\sum_{l\sim L}\sum_{\mu y/2L<d\leq X/L}\frac{\Lambda(l)a(d)\chi(dl)}{(dl^2)^{\varkappa+i\gamma_6}}\right|\nonumber\\
\end{align*}

\begin{align}\label{0I6est1}
&\ll\left|\sum_{l\sim L}\frac{\Lambda(l) \chi(l)}{l^{2\varkappa+i2\gamma_5}\log\left(\frac{lX}{L}\right)}\right|\Bigg|\sum_{\mu y/2L<d\leq X/L}\frac{a(d)\chi(d)}{d^{\varkappa+i\gamma_5}}\Bigg|\nonumber\\
&+\left|\sum_{l\sim L}\frac{\Lambda(l) \chi(l)}{l^{2\varkappa+i2\gamma_5}\log^2(z_0l)}\right|\Bigg|\sum_{\mu y/2L<d\leq z_0}\frac{a(d)\chi(d)}{d^{\varkappa+i\gamma_5}}\Bigg|\log X\nonumber\\
&+\left|\sum_{l\sim L}\sum_{\mu y/2L<d\leq X/L}\frac{\Lambda(l)a(d)\chi(dl)}{(dl^2)^{\varkappa+i\gamma_6}}\right|\,,
\end{align}
for some $|\gamma_5|, |\gamma_6|\leq T$ and for some $z_0\in[\mu y/2L, X/L]$.
Summarizing \eqref{0U4est1}, \eqref{0mathfrakX1est1}, \eqref{0mathfrakX2est1}, \eqref{0I1est1}, \eqref{0I23est1}, \eqref{0I4est1}, \eqref{0I5est1}, \eqref{0I5est2}, \eqref{0I6est1} we find
\begin{align}\label{0U4est2}
U_4(y,\chi,t)&\ll \left|\sum_{l\sim L}\frac{\Lambda(l) \chi(l)}{l^{\varkappa+i\gamma_1}}\right|
\Bigg|\sum_{\mu y/2L<d\leq X/L}\frac{a(d)\chi(d)}{d^{\varkappa+i\gamma_1}}\Bigg|\log^2X+X^{\frac{1}{4}}u^{-1}\log^2X+\log^2X\nonumber\\
&\ll\left|\sum_{l\sim L}\frac{\Lambda(l) \chi(l)}{l^{2\varkappa+i2\gamma_5}\log\left(\frac{lX}{L}\right)}\right|\Bigg|\sum_{\mu y/2L<d\leq X/L}\frac{a(d)\chi(d)}{d^{\varkappa+i\gamma_5}}\Bigg|\log X\nonumber\\
&+\left|\sum_{l\sim L}\frac{\Lambda(l) \chi(l)}{l^{2\varkappa+i2\gamma_5}\log^2(z_0l)}\right|\Bigg|\sum_{\mu y/2L<d\leq z_0}\frac{a(d)\chi(d)}{d^{\varkappa+i\gamma_5}}\Bigg|\log^2X\nonumber\\
&+\left|\sum_{l\sim L}\frac{\Lambda(l) \chi(l)}{l^{2\varkappa+i2\gamma_7}}\right|\Bigg|\sum_{\mu y/2L<d\leq X/L}\frac{a(d)\chi(d)}{d^{\varkappa+i\gamma_7}}\Bigg|\log^2X
\end{align}
for some $|\gamma_1|, |\gamma_5|, |\gamma_7|\leq T$.
Now \eqref{0Omegaj} and  \eqref{0U4est2} yield
\begin{equation}\label{0Omega24est1}
\Omega^{(4)}_2\ll \big(\Xi_1+\Xi_2+\Xi_3+\Xi_4+RX^{\frac{1}{4}} u^{-1}+R\big)\log^2X\,,
\end{equation}
where
\begin{align}
\label{0Xi1}
\Xi_1&=\sum\limits_{R_0<r\le R}\frac{1}{\varphi(r)}\sideset{}{^*}\sum\limits_{\chi(r)}
\left|\sum_{l\sim L}\frac{\Lambda(l) \chi(l)}{l^{\varkappa+i\gamma_1}}\right|\Bigg|\sum_{\mu y/l<d\leq X/L}\frac{a(d)\chi(d)}{d^{{\varkappa+i\gamma_1}}}\Bigg|\,,\\
\label{0Xi2}
\Xi_2&=\sum\limits_{R_0<r\le R}\frac{1}{\varphi(r)}\sideset{}{^*}\sum\limits_{\chi(r)}
\left|\sum_{l\sim L}\frac{\Lambda(l) \chi(l)}{l^{2\varkappa+i2\gamma_5}\log\left(\frac{lX}{L}\right)}\right|\Bigg|\sum_{\mu y/2L<d\leq X/L}\frac{a(d)\chi(d)}{d^{\varkappa+i\gamma_5}}\Bigg|\,,\\
\label{0Xi3}
\Xi_3&=\sum\limits_{R_0<r\le R}\frac{1}{\varphi(r)}\sideset{}{^*}\sum\limits_{\chi(r)}
\left|\sum_{l\sim L}\frac{\Lambda(l) \chi(l)}{l^{2\varkappa+i2\gamma_5}\log^2(z_0l)}\right|\Bigg|\sum_{\mu y/2L<d\leq z_0}\frac{a(d)\chi(d)}{d^{\varkappa+i\gamma_5}}\Bigg|\,,\\
\label{0Xi4}
\Xi_4&=\sum\limits_{R_0<r\le R}\frac{1}{\varphi(r)}\sideset{}{^*}\sum\limits_{\chi(r)}
\left|\sum_{l\sim L}\frac{\Lambda(l) \chi(l)}{l^{2\varkappa+i2\gamma_7}}\right|\Bigg|\sum_{\mu y/2L<d\leq X/L}\frac{a(d)\chi(d)}{d^{\varkappa+i\gamma_7}}\Bigg|\,.
\end{align}
By \eqref{0R0R}, \eqref{0cdad}, \eqref{0kappaT},  Cauchy's inequality, Lemma \ref{largesieve} and the well known inequalities 
\begin{equation*}
\sum\limits_{n\le X} \Lambda^2(n)\ll X\log X \,, \quad \sum\limits_{n\le X} \tau^2(n)\ll X\log^3X    
\end{equation*}
we derive
\begin{align}\label{0Xi1helpest}
&\sum\limits_{R_0<r\le R}
\frac{r}{\varphi(r)}\sideset{}{^*}\sum\limits_{\chi(r)}\left|\sum_{l\sim L}\frac{\Lambda(l) \chi(l)}{l^{\varkappa+i\gamma_1}}\right|
\Bigg|\sum_{\mu y/2L<d\leq X/L}\frac{a(d)\chi(d)}{d^{{\varkappa+i\gamma_1}}}\Bigg|\nonumber\\
&\ll\left(\sum\limits_{R_0<r\le R}\frac{r}{\varphi(r)}\sideset{}{^*}\sum\limits_{\chi(r)}
\left|\sum_{l\sim L}\frac{\Lambda(l) \chi(l)}{l^{\varkappa+i\gamma_1}}\right|^2\right)^{\frac{1}{2}}\nonumber\\
&\times\left(\sum\limits_{R_0<r\le R}\frac{r}{\varphi(r)}\sideset{}{^*}\sum\limits_{\chi(r)}
\Bigg|\sum_{\mu y/2L<d\leq X/L}\frac{a(d)\chi(d)}{d^{{\varkappa+i\gamma_1}}}\Bigg|^2\right)^{\frac{1}{2}}\nonumber\\
&\ll\big(L+ R^2\big)^{\frac{1}{2}}\left(\frac{X}{L}+ R^2\right)^{\frac{1}{2}}
\left(\sum_{l\sim L}\Lambda^2(l)\right)^{\frac{1}{2}}\left(\sum_{\mu y/2L<d\leq X/L}\tau^2(d)\right)^{\frac{1}{2}}\nonumber\\
&\ll\big(X + XRu^{-\frac{1}{2}}+X^{\frac{1}{2}}R^2\big)\log^2X\,.
\end{align}
Now \eqref{0R0R}, \eqref{0Xi1}, \eqref{0Xi1helpest} and Abel's summation formula lead to
\begin{equation}\label{0Xi1est}
\Xi_1\ll \big(XR_0^{-1} + Xu^{-\frac{1}{2}}\log X+X^{\frac{1}{2}}R\big)\log^2X\,.
\end{equation}
Proceeding in the same way for the sums \eqref{0Xi2} -- \eqref{0Xi4} we deduce
\begin{equation}\label{0Xi234est}
\Xi_2,\, \Xi_3,\, \Xi_4\ll \big(XR_0^{-1} + Xu^{-\frac{1}{2}}\log X+X^{\frac{1}{2}}R\big)\log^2X\,.
\end{equation}
Bearing in mind \eqref{0Omega24est1}, \eqref{0Xi1est} and \eqref{0Xi234est} we obtain
\begin{equation}\label{0Omega24est2}
\Omega^{(4)}_2\ll\big(XR_0^{-1} + Xu^{-\frac{1}{2}}\log X+X^{\frac{1}{2}}R\big)\log^4X\,.
\end{equation}
Arguing as with $\Omega^{(4)}_2$ we get
\begin{equation}\label{0Omega23est}
\Omega^{(3)}_2\ll\big(XR_0^{-1} + Xu^{-\frac{1}{2}}\log X+X^{\frac{1}{2}}R\big)\log^4X\,.
\end{equation}
From  \eqref{0Omega2est1}, \eqref{0Omega21est}, \eqref{0Omega22est}, \eqref{0Omega24est2} and \eqref{0Omega23est} we find
\begin{equation}\label{0Omega2est2}
\Omega_2\ll X^{\frac{1}{4}} u R^{\frac{3}{2}}\log^2X+\big(XR_0^{-1} + Xu^{-\frac{1}{2}}\log X+X^{\frac{1}{2}}R\big)\log^4X\,.
\end{equation}
Using \eqref{0Sigma''primitive}, \eqref{0Omega1est} and \eqref{0Omega2est2} we derive
\begin{equation}\label{0Sigma''primitiveest}
\Sigma''\ll X^{\frac{1}{4}} u R^{\frac{3}{2}}\log^3X+\big(XR_0^{-1} + Xu^{-\frac{1}{2}}\log X+X^{\frac{1}{2}}R\big)\log^5X+\frac{X}{\log^AX}\,.
\end{equation}
Apparently for each $B>0$, the theorem is true for $y=\log^BX$. Therefore we can choose  
\begin{equation}\label{0u}
u=(\log X)^{2A+12}
\end{equation}
Taking into account \eqref{0Sigmadecomp}, \eqref{0Sigma'est}, \eqref{0R0R}, \eqref{0Sigma''primitiveest} and \eqref{0u} we deduce
\begin{equation}\label{0Sigmaest}
\Sigma\ll\frac{X}{\log^AX}\,.
\end{equation}
Summarizing \eqref{Eytda}, \eqref{0Lambdalog}, \eqref{0Sigma} and \eqref{0Sigmaest} we establish Lemma \ref{Bomb-Vin-Dim}.

\end{proof}

Put
\begin{align}
\label{S1}
&S_1=S(t)\,,\\
\label{S2}
&S_2=S_{l,d;J}(t)\,,\\
\label{I1}
&I_1=I(t)\,,\\
\label{I2}
&I_2=\frac{I_J(t)}{\varphi(d)}\,.
\end{align}
We use the identity
\begin{equation}\label{Identity}
S^2_1S_2=I^2_1I_2+(S_2-I_2)I_1^2+S_2(S_1-I_1)I_1+S_1S_2(S_1-I_1)\,.
\end{equation}
Define
\begin{align}
\label{PhiDeltaJX}
&\Phi_{\Delta, J}(X, d)=\frac{1}{\varphi(d)}
\int\limits_{-\Delta}^{\Delta}\Theta(t)I^2(t)I_J(t)e(-Nt)\,dt\,,\\
\label{PhiJX}
&\Phi_J(X, d)=\frac{1}{\varphi(d)}\int\limits_{-\infty}^{\infty}\Theta(t)I^2(t)I_J(t)e(-Nt)\,dt\,.
\end{align}
Now  \eqref{SalphaX}, \eqref{IalphaX}, \eqref{Ild1}, \eqref{S1} -- \eqref{PhiDeltaJX},
Lemma \ref{Fourier}, Lemma \ref{SIasympt}, Lemma \ref{intLintI},
Lemma \ref{intSld} and  Cauchy's inequality imply
\begin{align*}
I^{(1)}_{l,d;J}(X)-\Phi_{\Delta, J}(X, d)
&=\int\limits_{-\Delta}^{\Delta}\Theta(t)\Bigg(S_{l,d;J}(t)-\frac{I_J(t)}{\varphi(d)}\Bigg)
I^2(t)e(-N t)\,dt\nonumber\\
&+\int\limits_{-\Delta}^{\Delta}\Theta(t)S_{l,d;J}(t)\Big(S(t)-I(t)\Big)I(t)e(-Nt)\,dt\nonumber\\
&+\int\limits_{-\Delta}^{\Delta}\Theta(t)S(t)S_{l,d;J}(t)\Big(S(t)-I(t)\Big)e(-Nt)\,dt\nonumber\\
\end{align*}
\begin{align}\label{Ild1-PhiDeltaJX}
&\ll\varepsilon  \left[ \max\limits_{|t|\leq\Delta}\bigg|S_{l,d;J}(t)-\frac{I_J(t)}{\varphi(d)}\bigg|
\int\limits_{-\Delta}^{\Delta}|I(t)|^2\,dt\right.\nonumber\\
&+\frac{X}{e^{(\log X)^{1/5}}}\Bigg(\int\limits_{-\Delta}^{\Delta}|S_{l,d;J}(t)|^2\,dt\Bigg)^{\frac{1}{2}}
\Bigg(\int\limits_{-\Delta}^{\Delta}|I(t)|^2\,dt)\Bigg)^{\frac{1}{2}}\nonumber\\
&\left.+\frac{X}{e^{(\log X)^{1/5}}}\Bigg(\int\limits_{-\Delta}^{\Delta}|S(t)|^2\,dt\Bigg)^{\frac{1}{2}}
\Bigg(\int\limits_{-\Delta}^{\Delta}|S_{l,d;J}(t)|^2\,dt\Bigg)^{\frac{1}{2}}\right]\nonumber\\
&\ll\varepsilon\Bigg(X^{2-c}(\log X)\max\limits_{|t|\leq\Delta}
\bigg|S_{l,d;J}(t)-\frac{I_J(t)}{\varphi(d)}\bigg|
+\frac{X^{3-c}}{de^{(\log X)^{1/6}}}\Bigg)\,.
\end{align}
Using  \eqref{IJalphaX}, \eqref{IalphaX} and Lemma \ref{IestTitchmarsh} we deduce
\begin{equation}\label{IalphaXest}
I_J(t)\ll \min \left(X, \,\frac{X^{1-c}}{|t|}\right)\,,  \quad
I(t)\ll \min \left(X, \,\frac{X^{1-c}}{|t|}\right)\,.
\end{equation}
From \eqref{IJalphaX}, \eqref{IalphaX}, \eqref{PhiDeltaJX}, \eqref{PhiJX}, \eqref{IalphaXest}
and Lemma \ref{Fourier} it follows
\begin{equation*}
\Phi_{\Delta, J}(X, d)-\Phi_J(X, d)\ll \frac{1}{\varphi(d)}
\int\limits_{\Delta}^{\infty} |I(t)|^2|I_J(t)| |\Theta(t)|\,dt
\ll \varepsilon\frac{ X^{3-3c}}{\varphi(d)} \int\limits_{\Delta}^{\infty} \frac{dt}{t^3}
\ll \frac{\varepsilon X^{3-3c}}{\varphi(d)\Delta^2}
\end{equation*}
and therefore
\begin{equation}\label{JXest1}
\Phi_{\Delta, J}(X, d)=\Phi_J(X, d)
+\mathcal{O}\left(\frac{\varepsilon X^{3-3c}}{\varphi(d)\Delta^2}\right)\,.
\end{equation}
Finally  \eqref{Delta}, \eqref{Ild1-PhiDeltaJX}, \eqref{JXest1} and the identity
\begin{equation*}
I^{(1)}_{l,d;J}(X)=I^{(1)}_{l,d;J}(X)-\Phi_{\Delta, J}(X, d)
+\Phi_{\Delta, J}(X, d)-\Phi_J(X, d)+\Phi_J(X, d)
\end{equation*}
yield
\begin{equation}\label{Ild1est}
I^{(1)}_{l,d;J}(X)=\Phi_J(X, d)
+\mathcal{O}\Bigg(\varepsilon X^{2-c}
(\log X)\max\limits_{|t|\leq\Delta}\bigg|S_{l,d;J}(t)-\frac{I_J(t)}{\varphi(d)}\bigg|\Bigg)
+\mathcal{O}\bigg(\frac{\varepsilon X^{3-c}}{de^{(\log X)^{1/6}}}\bigg)\Bigg)\,.
\end{equation}

\section{Upper bound  of $\mathbf{I^{(3)}_{l,d;J}(X)}$}\label{SectionIld3}
\indent

By \eqref{H}, \eqref{SldalphaX}, \eqref{SalphaX}, \eqref{Ild3} and Lemma \ref{Fourier} we find
\begin{equation}\label{Ild3est}
I^{(3)}_{l,d;J}(X)\ll
\frac{X^3\log X}{d}\int\limits_{H}^{\infty}\frac{1}{t}\bigg(\frac{k}{2\pi\delta t}\bigg)^k \,dt
=\frac{X^3\log X}{dk}\bigg(\frac{k}{2\pi\delta H}\bigg)^k
\ll\frac{1}{d}\,.
\end{equation}

\section{Upper bound of $\mathbf{\Gamma_3(X)}$}\label{SectionGamma3}
\indent

Consider the sum  $\Gamma_3(X)$.\\
Since
\begin{equation*}
\sum\limits_{d|p_1-1\atop{d\geq X/D}}\chi_4(d)=\sum\limits_{m|p_1-1\atop{m\leq (p_1-1)D/X}}
\chi_4\bigg(\frac{p_1-1}{m}\bigg)
=\sum\limits_{j=\pm1}\chi_4(j)\sum\limits_{m|p_1-1\atop{m\leq (p_1-1)D/X
\atop{\frac{p_1-1}{m}\equiv j\,(4)}}}1
\end{equation*}
then from \eqref{Gamma3} and \eqref{Ild} we obtain
\begin{equation*}
\Gamma_3(X)=\sum\limits_{m<D\atop{2|m}}\sum\limits_{j=\pm1}\chi_4(j)I_{1+jm,4m;J_m}(X)\,,
\end{equation*}
where $J_m=\big(\max\{1+mX/D,X/2\},X\big]$.
The last formula and \eqref{Ilddecomp} give us
\begin{equation}\label{Gamma3decomp}
\Gamma_3(X)=\Gamma_3^{(1)}(X)+\Gamma_3^{(2)}(X)+\Gamma_3^{(3)}(X)\,,
\end{equation}
where
\begin{equation}\label{Gamma3i}
\Gamma_3^{(i)}(X)=\sum\limits_{m<D\atop{2|m}}\sum\limits_{j=\pm1}\chi_4(j)
I_{1+jm,4m;J_m}^{(i)}(X)\,,\;\; i=1,\,2,\,3.
\end{equation}

\subsection{Estimation of $\mathbf{\Gamma_3^{(1)}(X)}$}
\indent

From \eqref{Ild1est} and \eqref{Gamma3i} we get
\begin{align}\label{Gamma31}
\Gamma_3^{(1)}(X)=\Gamma^*&
+\mathcal{O}\Big(\varepsilon X^{2-c}(\log X)\Sigma_1\Big)
+\mathcal{O}\bigg(\frac{\varepsilon X^{3-c}}{e^{(\log X)^{1/6}}}\Sigma_2\bigg)\,,
\end{align}
where
\begin{align}
\label{Gamma*}
&\Gamma^*=\sum\limits_{m<D\atop{2|m}}
\Phi_J(X, 4m)\sum\limits_{j=\pm1}\chi_4(j)\,,\\
\label{Sigma1}
&\Sigma_1=\sum\limits_{m<D\atop{2|m}}
\max\limits_{|t|\leq\Delta}\bigg|S_{1+jm,4m;J}(t)-\frac{I_J(t)}{\varphi(4m)}\bigg|\,,\\
\label{Sigma2}
&\Sigma_2=\sum\limits_{m<D}\frac{1}{4m}\,.
\end{align}
From the properties of $\chi_4(k)$ we have that
\begin{equation}\label{Gamma*est}
\Gamma^*=0\,.
\end{equation}
By \eqref{D}, \eqref{SldalphaX},  \eqref{IJalphaX}, \eqref{Sigma1} and Lemma \ref{Bomb-Vin-Dim} we deduce
\begin{equation}\label{Sigma1est}
\Sigma_1\ll\frac{X}{\log^AX}\,.
\end{equation}
It is well known that
\begin{equation}\label{Sigma2est}
\Sigma_2\ll \log X\,.
\end{equation}
Bearing in mind \eqref{Gamma31}, \eqref{Gamma*est}, \eqref{Sigma1est} and \eqref{Sigma2est} we find
\begin{equation}\label{Gamma31est}
\Gamma_3^{(1)}(X)\ll\frac{\varepsilon X^{3-c}}{\log X}\,.
\end{equation}

\subsection{Estimation of $\mathbf{\Gamma_3^{(2)}(X)}$}
\indent

Now we consider $\Gamma_3^{(2)}(X)$. From \eqref{Ild2} and \eqref{Gamma3i} we have
\begin{equation}\label{Gamma32}
\Gamma_3^{(2)}(X)=\int\limits_{\Delta\leq|t|\leq H}\Theta(t)S^2(t)K(t)e(-Nt)\,dt\,,
\end{equation}
where
\begin{equation}\label{Kt}
K(t)=\sum\limits_{m<D\atop{2|m}}\sum\limits_{j=\pm1}\chi_4(j)S_{1+jm,4m;J_m}(t)\,.
\end{equation}

\begin{lemma}\label{SIest} Assume that
\begin{equation}\label{Conditions1}
\Delta \leq |t| \leq H\,, \quad |a(m)|\ll m^\eta \,,\quad LM\asymp X\,,\quad L\gg X^{\frac{2}{5}} \,.
\end{equation}
Set
\begin{equation}\label{SI}
S_I=\sum\limits_{m\sim M}a(m)\sum\limits_{l\sim L}e(tm^cl^c)\,.
\end{equation}
Then
\begin{equation*}
S_I\ll X^{\frac{373}{400}+\eta}\,.
\end{equation*}
\end{lemma}
\begin{proof}
We first consider the case when
\begin{equation}\label{M411}
M\ll X^{\frac{4}{11}}.
\end{equation}
By \eqref{Delta}, \eqref{H}, \eqref{Conditions1}, \eqref{SI}, \eqref{M411} and Lemma \ref{Exponentpairs} with the exponent pair $\left(\frac{1}{2},\frac{1}{2}\right)$ we obtain

\begin{align}\label{SIest1}
S_I&\ll X^\eta\sum\limits_{m\sim M}\bigg|\sum\limits_{l\sim L}e(tm^cl^c)\bigg|\ll X^\eta\sum\limits_{m\sim M}\Bigg(\big(|t|X^cL^{-1}\big)^{\frac{1}{2}}L^{\frac{1}{2}}+\frac{1}{|t|X^cL^{-1}}\Bigg)\nonumber\\
&\ll X^\eta\Big(MH^{\frac{1}{2}}X^{\frac{c}{2}}+\Delta^{-1}X^{1-c}\Big)\ll X^\eta\Big(MX^{\frac{c}{2}}+X^{\frac{3}{4}}\Big)\ll X^{\frac{373}{400}+\eta}\,.
\end{align}
Next we consider the case when
\begin{equation}\label{M41135}
X^{\frac{4}{11}}\ll M\ll X^{\frac{3}{5}}\,.
\end{equation}
Using \eqref{SI}, \eqref{M41135} and Lemma \ref{Baker-Weingartnerest} we deduce
\begin{equation}\label{SIest2}
S_I\ll X^{\frac{373}{400}+\eta}\,.
\end{equation}
Bearing in mind \eqref{SIest1} and \eqref{SIest2} we establish the statement in the lemma.
\end{proof}

\begin{lemma}\label{SIIest} Assume that
\begin{equation}\label{Conditions2}
\Delta \leq |t| \leq H\,, \quad |a(m)|\ll m^\eta \,, \quad |b(l)|\ll l^\eta\,,
\quad LM\asymp X\,,\quad X^{\frac{1}{5}} \ll L\ll X^{\frac{1}{3}} \,.
\end{equation}
Set
\begin{equation}\label{SII}
S_{II}=\sum\limits_{m\sim M}a(m)\sum\limits_{l\sim L}b(l)e(tm^cl^c)\,.
\end{equation}
Then
\begin{equation*}
S_{II}\ll X^{\frac{373}{400}+\eta}\,.
\end{equation*}
\end{lemma}
\begin{proof}
We first consider the case when
\begin{equation}\label{L151960}
X^{\frac{1}{5}} \ll L\ll X^{\frac{63}{200}}\,.
\end{equation}
From \eqref{Conditions2}, \eqref{SII}, Cauchy's inequality and Lemma \ref{Squareout} with $Q=X^{\frac{1}{5}}$ it follows that
\begin{equation}\label{SIIest1}
|S_{II}|^2\ll X^\eta\Bigg(\frac{X^2}{Q}+\frac{X}{Q}\sum\limits_{1\leq q\leq Q}
\sum\limits_{l\sim L}\bigg|\sum\limits_{m\sim M}e\big(f(l, m, q)\big)\bigg|\Bigg)\,,
\end{equation}
where $f(l, m, q)=tm^c\big((l+q)^c-l^c\big)$.
Now \eqref{Delta}, \eqref{H}, \eqref{Conditions2}, \eqref{L151960}, \eqref{SIIest1} and Lemma \ref{Exponentpairs} with the exponent pair
$\left(\frac{2}{7},\frac{4}{7}\right)$ give us
\begin{align}\label{SIIest2}
S_{II}&\ll X^\eta\Bigg(\frac{X^2}{Q}+\frac{X}{Q}\sum\limits_{1\leq q\leq Q}
\sum\limits_{l\sim L}\bigg(\big(|t|qX^{c-1}\big)^{\frac{2}{7}}M^{\frac{4}{7}}+\frac{1}{|t|qX^{c-1}}\bigg)\Bigg)^{\frac{1}{2}}\nonumber\\
&\ll X^\eta\Bigg(\frac{X^2}{Q}+\frac{X}{Q}
\bigg(H^{\frac{2}{7}}X^{\frac{2(c-1)}{7}}M^{\frac{4}{7}}Q^{\frac{9}{7}}L+\Delta^{-1}X^{1-c}L\log Q\bigg)\Bigg)^{\frac{1}{2}}\nonumber\\
&\ll X^{\frac{373}{400}+\eta}\,.
\end{align}
Next we consider the case when
\begin{equation}\label{L19613}
X^{\frac{63}{200}}\ll L\ll X^{\frac{1}{3}}\,.
\end{equation}
Using \eqref{SII}, \eqref{L19613} and Lemma \ref{Sargos-Wuest} we find
\begin{equation}\label{SIIest3}
S_{II}\ll X^{\frac{373}{400}+\eta}\,.
\end{equation}
Taking into account \eqref{SIIest2} and \eqref{SIIest3} we establish the statement in the lemma.
\end{proof}

\begin{lemma}\label{SalphaXest} Let $\Delta \leq |t| \leq H$.
Then  for the exponential sum denoted by \eqref{SalphaX} we have
\begin{equation*}
S(t)\ll X^{\frac{373}{400}+\eta}\,.
\end{equation*}
\end{lemma}
\begin{proof}
In order to prove the lemma  we will use the formula
\begin{equation}\label{Lambdalog2}
S(t)=S^\ast(t)+\mathcal{O}\Big(X^{\frac{1}{2}+\varepsilon}\Big)\,,
\end{equation}
where
\begin{equation*}
S^\ast(t)=\sum\limits_{X/2<n\leq X}\Lambda(n)e(t n^c)\,.
\end{equation*}
Let
\begin{equation*}
U=X^{\frac{1}{5}}\,,\quad V=X^{\frac{1}{3}}\,,\quad Z=\big[X^{\frac{2}{5}}\big]+\frac{1}{2}\,.
\end{equation*}
According to Lemma \ref{Heath-Brown}, the sum $S^\ast(t)$
can be decomposed into $O\Big(\log^{10}X\Big)$ sums, each of which is either of Type I
\begin{equation*}
\sum\limits_{m\sim M}a(m)\sum\limits_{l\sim L}e(tm^cl^c)\,,
\end{equation*}
where
\begin{equation*}
L \gg Z\,, \quad  LM\asymp X\,, \quad |a(m)|\ll m^\eta\,,
\end{equation*}
or of Type II
\begin{equation*}
\sum\limits_{m\sim M}a(m)\sum\limits_{l\sim L}b(l)e(tm^cl^c)\,,
\end{equation*}
where
\begin{equation*}
U \ll L \ll V\,, \quad  LM\asymp X\,, \quad |a(m)|\ll m^\eta\,,\quad |b(l)|\ll l^\eta\,.
\end{equation*}
Using Lemma \ref{SIest} and  Lemma \ref{SIIest} we deduce
\begin{equation}\label{Sast}
S^\ast(t)\ll X^{\frac{373}{400}+\eta}\,.
\end{equation}
Bearing in mind \eqref{Lambdalog2} and  \eqref{Sast} we establish the statement in the lemma.
\end{proof}

\begin{lemma}\label{IntK2}
For the sum denoted by \eqref{Kt} we have
\begin{equation*}
\int\limits_{\Delta}^{H}|K(t)|^2|\Theta(t)|\,dt\ll X\log^7 X\,.
\end{equation*}
\end{lemma}
\begin{proof}
By Lemma \ref{Fourier} we get
\begin{align}\label{IntK2est1}
\int\limits_{\Delta}^{H}|K(t)|^2|\Theta(t)|\,dt
&\ll\varepsilon\int\limits_{\Delta}^{1/\varepsilon}|K(t)|^2dt
+\int\limits_{1/\varepsilon}^{H}\frac{|K(t)|^2}{t}\,dt\nonumber\\
&\ll\varepsilon\sum\limits_{0\leq n\leq1/\varepsilon}\int\limits_{n}^{n+1}|K(t)|^2\,dt
+\sum\limits_{1/\varepsilon-1\leq n\leq H}\frac{1}{n}\int\limits_{n}^{n+1}|K(t)|^2\,dt\,.
\end{align}
On the other hand  \eqref{SldalphaX} and \eqref{Kt} yield

\begin{align}\label{IntK1}
\int\limits_{n}^{n+1}|K(t)|^2\,dt&=
\sum\limits_{m_1,m_2<D\atop{2\mid m_1,2\mid m_2}}\sum\limits_{j_1=\pm1\atop{j_2=\pm1}}\chi_4(j_1)\chi_4(j_2)\nonumber\\
&\times\int\limits_{n}^{n+1}S_{ 1+j_1m_1, 4m_1;J_{m_1}}(t)S_{ 1+j_2m_2, 4m_2;J_{m_2}}(-t)\,dt\nonumber\\
&=\sum\limits_{m_1,m_2<D\atop{2\mid m_1,2\mid m_2}}\sum\limits_{j_1=\pm1\atop{j_2=\pm1}}\chi_4(j_1)\chi_4(j_2)\nonumber\\
&\times\sum\limits_{p_i\in J_{m_i},i=1,2\atop{p_i\equiv 1+j_im_i\, (4m_i), i=1,2}}
\log p_1\log p_2\int\limits_{n}^{n+1}e\big((p_1^c-p_2^c)t\big)\,dt\nonumber\\
&\ll\sum\limits_{m_i\leq D\atop{i=1,2}}
\sum\limits_{X/2<p_1,p_2\leq X\atop{p_i\equiv 1+j_im_i\,(4m_i), i=1,2}}\log p_1\log p_2
\min\bigg(1,\frac{1}{|p_1^c-p_2^c|}\bigg)\nonumber\\
&\ll(\log X)^2\sum\limits_{m_i\leq D\atop{i=1,2}}
\sum\limits_{X/2<n_1,n_2\leq X\atop{n_i\equiv 1+j_im_i\,(4m_i), i=1,2}}
\min\bigg(1,\frac{1}{|n_1^c-n_2^c|}\bigg)\nonumber\\
&=(\log X)^2\sum\limits_{X/2<n_1,n_2\leq X}
\min\bigg(1,\frac{1}{|n_1^c-n_2^c|}\bigg)
\sum\limits_{m_1\leq D\atop{4m_1|n_1-1-j_1m_1}}1
\sum\limits_{m_2\leq D\atop{4m_2|n_2-1-j_2m_2}}1\nonumber\\
&\ll(\log X)^2\sum\limits_{X/2<n_1,n_2\leq X}
\min\bigg(1,\frac{1}{|n_1^c-n_2^c|}\bigg)
\tau(n_1-1)\tau(n_2-1)\nonumber\\
&\ll(\log X)^2\sum\limits_{X/2<n_1,n_2\leq X}
\tau^2(n_1-1)\min\bigg(1,\frac{1}{|n_1^c-n_2^c|}\bigg)\nonumber\\
&\ll(\mathfrak{S}_1+\mathfrak{S}_2)\log^2X\,,
\end{align}
where
\begin{equation*}
\mathfrak{S}_1=\sum\limits_{X/2<n_1,n_2\leq X\atop{|n_1^c-n_2^c|\leq1}}\tau^2(n_1-1)\,,\quad
\mathfrak{S}_2=\sum\limits_{X/2<n_1,n_2\leq X\atop{|n_1^c-n_2^c|>1}}\frac{\tau^2(n_1-1)}{|n_1^c-n_2^c|}\,.
\end{equation*}
First we shall estimate $\mathfrak{S}_1$.
By the mean-value theorem and \eqref{Lambdatauest2} we find
\begin{align}\label{mathfrakS1est}
\mathfrak{S}_1&=\sum\limits_{X/2<n_1\leq X}\tau^2(n_1-1)
\sum\limits_{X/2<n_2\leq X\atop{{(n_1^c-1)^{1/c}\leq n_2\leq(n_1^c+1)^{1/c}}}}1\nonumber\\
&\ll\sum\limits_{X/2<n\leq X}\tau^2(n-1)\big((n^c+1)^{1/c}-(n^c-1)^{1/c}+1\big)\nonumber\\
&\ll\sum\limits_{X/2<n\leq X}\tau^2(n-1)\big(X^{1-c}+1\big)\nonumber\\
&\ll\sum\limits_{X/2<n\leq X}\tau^2(n-1)\nonumber\\
&\ll X\log^3X\,.
\end{align}
Now we consider $\mathfrak{S}_2$. We have
\begin{align}\label{mathfrakS2est1}
\mathfrak{S}_2&\ll\sum\limits_l\sum\limits_{X/2<n_1,n_2\leq X
\atop{l<|n_1^c-n_2^c|\leq2l}}\frac{\tau^2(n_1-1)}{|n_1^c-n_2^c|}\nonumber\\
&\ll\sum\limits_l\frac{1}{l}\sum\limits_{X/2<n_1\leq X}
\tau^2(n_1-1)U(n_1,l)\,,
\end{align}
where
\begin{equation}\label{llogX}
l\ll\log X
\end{equation}
and
\begin{equation*}
U(n_1,l)=\sum\limits_{X/2<n_2\leq X\atop{(n_1^c+l)^{1/c}\leq n_2\leq (n_1^c+2l)^{1/c}}}1\,.
\end{equation*}
By the mean-value theorem and \eqref{llogX} we deduce
\begin{equation}\label{Un1l}
U(n_1,l)\ll(n_1^c+2l)^{1/c}-(n_1^c+l)^{1/c}+1\ll lX^{1-c}+1\ll1\,.
\end{equation}
Bearing in mind \eqref{Lambdatauest2} and \eqref{mathfrakS2est1} -- \eqref{Un1l} we obtain
\begin{equation}\label{mathfrakS2est2}
\mathfrak{S}_2\ll(\log X)\sum\limits_{X/2<n\leq X}\tau^2(n-1)\ll X\log^4X\,.
\end{equation}
The assertion in the lemma  follows from \eqref{IntK2est1}, \eqref{IntK1}, \eqref{mathfrakS1est}) and \eqref{mathfrakS2est2}.
\end{proof}
Using \eqref{Gamma32} and Cauchy's inequality we write
\begin{equation}\label{Gamma32est1}
\Gamma_3^{(2)}(X)\ll\max\limits_{\Delta\leq t\leq H}|S(t)|
\left(\int\limits_{\Delta}^{H}|S(t)|^2|\Theta(t)|\,dt\right)^{\frac{1}{2}}
\left(\int\limits_{\Delta}^{H}|K(t)|^2|\Theta(t)|\,dt\right)^{\frac{1}{2}}\,.
\end{equation}
According to Lemma \ref{Fourier} and Lemma \ref{intLintI} (iii)  we find
\begin{align}\label{IntS2Thetaest}
\int\limits_{\Delta}^{H}|S(t)|^2|\Theta(t)|\,dt
&\ll\varepsilon\int\limits_{\Delta}^{1/\varepsilon}|S(t)|^2dt
+\int\limits_{1/\varepsilon}^{H}\frac{|S(t)|^2}{t}\,dt\nonumber\\
&\ll\varepsilon\sum\limits_{0\leq n\leq1/\varepsilon}\int\limits_{n}^{n+1}|S(t)|^2\,dt
+\sum\limits_{1/\varepsilon-1\leq n\leq H}\frac{1}{n}\int\limits_{n}^{n+1}|S(t)|^2\,dt\nonumber\\
&\ll X\log^4X\,.
\end{align}
Finally \eqref{varepsilon}, \eqref{Gamma32est1},  \eqref{IntS2Thetaest},
Lemma \ref{SalphaXest} and Lemma \ref{IntK2} imply
\begin{equation}\label{Gamma32est2}
\Gamma_3^{(2)}(X)\ll\frac{\varepsilon X^{3-c}}{\log X}\,.
\end{equation}

\subsection{Estimation of $\mathbf{\Gamma_3^{(3)}(X)}$}
\indent

From \eqref{Ild3est} and \eqref{Gamma3i} we have
\begin{equation}\label{Gamma33est}
\Gamma_3^{(3)}(X)\ll\sum\limits_{m<D}\frac{1}{d}\ll \log X\,.
\end{equation}

\subsection{Estimation of $\mathbf{\Gamma_3(X)}$}
\indent

Summarizing \eqref{Gamma3decomp}, \eqref{Gamma31est}, \eqref{Gamma32est2} and \eqref{Gamma33est} we get
\begin{equation}\label{Gamm3est}
\Gamma_3(X)\ll\frac{\varepsilon X^{3-c}}{\log X}\,.
\end{equation}

\section{Upper bound of $\mathbf{\Gamma_2(X)}$}\label{SectionGamma2}
\indent

In this section we need a lemma that gives us information about the upper bound
of the number of solutions of the binary Piatetski-Shapiro inequality.
\begin{lemma}\label{Thenumberofsolutions}
Let $1<c<3$, $c\neq2$  and $N_0$ is a sufficiently large positive number.
Then for the number of solutions $B_0(N_0)$ of the  diophantine inequality
\begin{equation}\label{Binary}
|p_1^c+p_2^c-N_0|<\varepsilon
\end{equation}
in prime numbers $p_1,\,p_2 \in \left(N_0^{\frac{1}{c}}/2\,,\, N_0^{\frac{1}{c}}\right]$ we have that
\begin{equation*}
B_0(N_0)\ll \frac{\varepsilon N_0^{\frac{2}{c}-1}}{\log^2N_0}\,.
\end{equation*}
\end{lemma}
\begin{proof}
Define
\begin{equation}\label{BX0}
B(X_0)=\sum\limits_{X_0/2<p_1,p_2\leq X_0\atop{|p_1^c+p_2^c-N_0|<\varepsilon}}
\log p_1\log p_2\,,
\end{equation}
where
\begin{equation}\label{X0}
X_0=N_0^{\frac{1}{c}}\,.
\end{equation}
Let us take  the parameters
\begin{equation*}
a_0 = \frac{5\varepsilon}{4}\,,\quad \delta_0=\frac{\varepsilon}{4}\,,\quad  k_0=[\log X_0]\,.
\end{equation*}
According to Lemma \ref{Fourier} there exists a function $\theta_0(y)$
which is $k_0$ times continuously differentiable and such that
\begin{align*}
&\theta_0(y)=1\quad\quad\quad\mbox{for }\quad\;\; |y|\leq \varepsilon\,;\\
&0<\theta_0(y)<1\quad\,\mbox{for}\quad\quad  \varepsilon <|y|< \frac{3\varepsilon}{2}\,;\\
&\theta_0(y)=0\quad\quad\quad\mbox{for}\quad\quad|y|\geq \frac{3\varepsilon}{2}
\end{align*}
and its Fourier transform
\begin{equation*}
\Theta_0(x)=\int\limits_{-\infty}^{\infty}\theta_0(y)e(-xy)dy
\end{equation*}
satisfies the inequality
\begin{equation}\label{Theta0est}
|\Theta_0(x)|\leq\min\Bigg(\frac{5\varepsilon}{2},\frac{1}{\pi|x|},\frac{1}{\pi |x|}
\bigg(\frac{2k_0}{\pi |x|\varepsilon}\bigg)^{k_0}\Bigg)\,.
\end{equation}
By \eqref{BX0}, the definition of $\theta_0(y)$ and  the inverse Fourier
transformation formula we use decomposition over major, minor and trivial arcs as follows
\begin{align}\label{BX0est1}
B(X_0)&\leq\sum\limits_{X_0/2<p_1,p_2\leq X_0}
\theta_0\big(p_1^c+p_2^c-N_0\big)\log p_1\log p_2\nonumber\\
&=\int\limits_{-\infty}^{\infty}\Theta_0(t) S_0^2(t) e(-N_0t)\,dt\nonumber\\
&=B_1(X_0)+B_2(X_0)+B_3(X_0)\,,
\end{align}
where
\begin{align}
\label{S0t}
&S_0(t)=\sum\limits_{X_0/2<p\leq X_0} e(t p^c)\log p\,,\\
\label{Delta0}
&\Delta_0=\frac{(\log X_0)^{A_0}}{X_0^c}\,,\quad A_0>10\,,\\
\label{B1}
&B_1(X_0)=\int\limits_{-\Delta_0}^{\Delta_0}\Theta_0(t)S_0^2(t)e(-N_0t)\,dt\,,\\
\label{B2}
&B_2(X_0)=\int\limits_{\Delta_0\leq|t|\leq H}\Theta_0(t)S_0^2(t)e(-N_0t)\,dt\,,\\
\label{B3}
&B_3(X_0)=\int\limits_{|t|>H}\Theta_0(t)S_0^2(t)e(-N_0 t)\,dt\,.
\end{align}
First we estimate $B_1(X_0)$.
Put
\begin{align}
\label{I0t}
&I_0(t)=\int\limits_{X_0/2}^{X_0}e(t y^c)\,dy\,,\\
\label{PsiDeltaX}
&\Psi_{\Delta_0}(X_0)=\int\limits_{-\Delta_0}^{\Delta_0}\Theta_0(t)I_0^2(t)e(-N_0t)\,dt\,,\\
\label{PsiX}
&\Psi(X_0)=\int\limits_{-\infty}^{\infty}\Theta_0(t)I_0^2(t)e(-N_0t)\,dt\,.
\end{align}
Using \eqref{Theta0est}, \eqref{I0t}, \eqref{PsiX} and Lemma \ref{IestTitchmarsh} we find
\begin{align}\label{Psiest}
\Psi(X_0)&=\int\limits_{-X_0^{-c}}^{X_0^{-c}}\Theta_0(t)I_0^2(t)e(-N_0t)\,dt
+\int\limits_{|t|>X_0^{-c}}\Theta_0(t)I_0^2(t)e(-N_0t)\,dt\nonumber\\
&\ll\int\limits_{-X_0^{-c}}^{X_0^{-c}}\varepsilon X^2_0\,dt
+\int\limits_{X_0^{-c}}^{\infty}\varepsilon \left(\frac{X_0^{1-c}}{t}\right)^2\,dt\nonumber\\
&\ll\varepsilon X_0^{2-c}\,.
\end{align}
On the other hand  \eqref{Theta0est}, \eqref{Delta0}, \eqref{B1}, \eqref{PsiDeltaX},
Lemma \ref{SIasympt} and the trivial estimations
\begin{equation}\label{S0I0trivest}
S_0(t)\ll X_0 \,, \quad I_0(t)\ll X_0
\end{equation}
give us
\begin{align}\label{B1PsiDelta}
B_1(X_0)-\Psi_{\Delta_0}(X_0)
&\ll\int\limits_{-\Delta_0}^{\Delta_0}|S_0^2(t)-I_0^2(t)||\Theta_0(t)|\,dt\nonumber\\
&\ll\varepsilon\int\limits_{-\Delta_0}^{\Delta_0}
\big|S_0(t)-I_0(t)\big|\Big(|S_0(t)|+|I_0(t)|\Big)\,dt\nonumber\\
&\ll\varepsilon \frac{X_0}{e^{(\log X_0)^{1/5}}}
\left(\int\limits_{-\Delta_0}^{\Delta_0}|S_0(t)|\,dt
+\int\limits_{-\Delta_0}^{\Delta_0}|I_0(t)|\,dt\right)\nonumber\\
&\ll\frac{\varepsilon X_0^{2-c}}{e^{(\log X_0)^{1/6}}}\,.
\end{align}
From \eqref{Theta0est}, \eqref{Delta0}, \eqref{PsiDeltaX}, \eqref{PsiX} and Lemma \ref{IestTitchmarsh} we  deduce
\begin{align}\label{PsiPsiDelta}
|\Psi(X_0)-\Psi_{\Delta_0}(X_0)|&\ll\int\limits_{\Delta_0}^{\infty}|I_0(t)|^2|\Theta_0(t)|\,dt
\ll\frac{\varepsilon}{X_0^{2(c-1)}}\int\limits_{\Delta_0}^{\infty}\frac{dt}{t^2}\nonumber\\
&\ll \frac{\varepsilon}{X_0^{2(c-1)}\Delta_0}\ll\frac{\varepsilon X_0^{2-c}}{\log X_0}\,.
\end{align}
Now \eqref{Psiest}, \eqref{B1PsiDelta} and \eqref{PsiPsiDelta} and the identity
\begin{equation*}
B_1(X_0)=B_1(X_0)-\Psi_{\Delta_0}(X_0)+\Psi_{\Delta_0}(X_0)-\Psi(X_0)+\Psi(X_0)
\end{equation*}
imply
\begin{equation}\label{B1X0est}
B_1(X_0)\ll \varepsilon X_0^{2-c}\,.
\end{equation}
Further we estimate $B_2(X_0)$. Consider the integral
\begin{equation}\label{B2'}
B^\ast_2(X_0)=\int\limits_{\Delta_0}^{H}\Theta_0(t)S_0^2(t)e(-N_0t)\,dt\,.
\end{equation}
By \eqref{X0}, \eqref{Theta0est}, \eqref{S0I0trivest}, \eqref{B2'} and partial integration it follows
\begin{align}\label{B2'est}
B^\ast_2(X_0)&=-\frac{1}{2\pi i}\int\limits_{\Delta_0}^{H}\frac{\Theta_0(t)S_0^2(t)}{N_0}\,d\,e(-N_0t)\nonumber\\
&=-\frac{\Theta_0(t)S_0^2(t)e(-N_0t)}{2\pi iN_0}\Bigg|_{\Delta_0}^{H}
+\frac{1}{2\pi iN_0}\int\limits_{\Delta_0}^{H}e(-N_0t)\,d\Big(\Theta_0(t)S_0^2(t)\Big)\nonumber\\
&\ll\varepsilon X_0^{2-c}+X_0^{-c}|\Omega|\,,
\end{align}
where
\begin{equation}\label{Omega}
\Omega=\int\limits_{\Delta_0}^{H}e(-N_0t)\,d\Big(\Theta_0(t)S_0^2(t)\Big)\,.
\end{equation}
Next we consider $\Omega$. Put
\begin{equation}\label{Gammat}
\Gamma_0 \, :\, z=g(t)=\Theta_0(t)S_0^2(t)\,,\quad g'(t)\neq0\,,\quad \Delta_0\leq t\leq H\,.
\end{equation}
Since $g(t)$ is a holomorphic function such that $g'(t)\neq0$ for $t\in [\Delta_0, H]$, then
there exists $g^{-1}(z)$  for $z\in \Gamma_0$.
Thus \eqref{Omega} and \eqref{Gammat} imply
\begin{equation}\label{Omegaest1}
\Omega=\int\limits_{\Gamma_0} e\Big(-N_0g^{-1}(z)\Big)\,dz\,.
\end{equation}
Using \eqref{Theta0est}, \eqref{S0I0trivest}, \eqref{Gammat} and that the integral \eqref{Omegaest1} is independent of path we obtain
\begin{equation}\label{Omegaest2}
\Omega=\int\limits_{\overline{\Gamma}_0} e\Big(-N_0g^{-1}(z)\Big)\,dz\ll\int\limits_{\overline{\Gamma}_0} |dz|
\ll |g(\Delta_0)|+|g(H)| \ll \varepsilon X_0^2\,,
\end{equation}
where $\overline{\Gamma}_0$ is the line segment connecting the points $g(\Delta_0)$ and $g(H)$.
Taking into account \eqref{B2}, \eqref{B2'}, \eqref{B2'est} and \eqref{Omegaest2} we deduce
\begin{equation}\label{B2X0est}
B_2(X_0)\ll \varepsilon X_0^{2-c}\,.
\end{equation}
Finally we estimate $B_3(X_0)$.
By \eqref{H}, \eqref{Theta0est}, \eqref{S0t}, \eqref{B3}, \eqref{S0I0trivest} we find
\begin{equation}\label{B3X0est}
B_3(X_0)\ll X_0^2\int\limits_{H}^{\infty}\frac{1}{t}\bigg(\frac{2k_0}{\pi t\varepsilon}\bigg)^{k_0} \,dt
\ll X_0^2\bigg(\frac{{k_0}}{\pi H\varepsilon}\bigg)^{k_0}\ll1\,.
\end{equation}
Summarizing \eqref{BX0est1}, \eqref{B1X0est}, \eqref{B2X0est} and \eqref{B3X0est} we get
\begin{equation}\label{BX0est}
B(X_0)\ll \varepsilon X_0^{2-c}\,.
\end{equation}
Bearing in mind \eqref{BX0}, \eqref{X0} and \eqref{BX0est}, for the number of solutions $B_0(N_0)$
of the  diophantine inequality \eqref{Binary} we obtain
\begin{equation*}
B_0(N_0)\ll \label{B1Psitau}\frac{\varepsilon N_0^{\frac{2}{c}-1}}{\log^2N_0}\,.
\end{equation*}
The lemma is proved.
\end{proof}
Consider the sum $\Gamma_2(X)$. We denote by $\mathcal{F}(X)$ the set of all primes
$X/2<p\leq X$ such that $p-1$ has a divisor belonging to the interval $(D,X/D)$.
The inequality $xy\leq x^2+y^2$ and  \eqref{Gamma2} give us
\begin{align*}
\Gamma_2(X)^2&\ll(\log X)^6\sum\limits_{X/2<p_1,...,p_6\leq X
\atop{|p_1^c+p_2^c+p_3^c-N|<\varepsilon
\atop{|p_4^c+p_5^c+p_6^c-N|<\varepsilon}}}
\left|\sum\limits_{d|p_1-1\atop{D<d<X/D}}\chi_4(d)\right|
\left|\sum\limits_{t|p_4-1\atop{D<t<X/D}}\chi_4(t)\right|\\
&\ll(\log X)^6\sum\limits_{X/2<p_1,...,p_6\leq X
\atop{|p_1^c+p_2^c+p_3^c-N|<\varepsilon
\atop{|p_4^c+p_5^c+p_6^c-N|<\varepsilon
\atop{p_4\in\mathcal{F}(X)}}}}\left|\sum\limits_{d|p_1-1\atop{D<d<X/D}}\chi_4(d)\right|^2\,.
\end{align*}
The summands in the last sum for which $p_1=p_4$ can be estimated with
$\mathcal{O}\big(X^{3+\varepsilon}\big)$.\\
Therefore
\begin{equation}\label{Gamma2est1}
\Gamma_2(X)^2\ll(\log X)^6\Sigma_0+X^{3+\varepsilon}\,,
\end{equation}
where
\begin{equation}\label{Sigma0}
\Sigma_0=\sum\limits_{X/2<p_1\leq X}
\left|\sum\limits_{d|p_1-1\atop{D<d<X/D}}\chi_4(d)\right|^2
\sum\limits_{X/2<p_4\leq X\atop{p_4\in\mathcal{F}(X)
\atop{p_4\neq p_1}}}\sum\limits_{X/2<p_2, p_3, p_5, p_6\leq X
\atop{|p_1^c+p_2^c+p_3^c-N|<\varepsilon
\atop{|p_4^c+p_5^c+p_6^c-N|<\varepsilon}}}1\,.
\end{equation}
Now  \eqref{Sigma0} and Lemma \ref{Thenumberofsolutions}  imply
\begin{equation}\label{Sigma0est}
\Sigma_0\ll \frac{X^{4-2c}}{\log^4X}\,\Sigma'_0\,\Sigma''_0\,,
\end{equation}
where
\begin{equation*}
\Sigma'_0=\sum\limits_{X/2<p\leq X}\left|\sum\limits_{d|p-1\atop{D<d<X/D}}\chi_4(d)\right|^2\,,
\quad \Sigma''_0=\sum\limits_{X/2<p\leq X\atop{p\in\mathcal{F}(X)}}1\,.
\end{equation*}
Applying Lemma \ref{Hooley1} we get
\begin{equation}\label{Sigma0'est}
\Sigma'_0\ll\frac{X(\log\log X)^7}{\log X}\,.
\end{equation}
Using Lemma \ref{Hooley2} we obtain
\begin{equation}\label{Sigma0''est}
\Sigma''_0\ll\frac{X(\log\log X)^3}{(\log X)^{1+2\theta_0}}\,,
\end{equation}
where $\theta_0$ is denoted by  \eqref{theta0}.

Finally \eqref{Gamma2est1}, \eqref{Sigma0est},
\eqref{Sigma0'est} and \eqref{Sigma0''est} yield
\begin{equation}\label{Gamma2est2}
\Gamma_2(X)\ll\frac{ X^{3-c}(\log\log X)^5}{(\log X)^{\theta_0}}
=\frac{\varepsilon X^{3-c}}{\log\log X}\,.
\end{equation}

\section{Lower bound  for $\mathbf{\Gamma_1(X)}$}\label{SectionGamma1}
\indent

Consider the sum $\Gamma_1(X)$.
From \eqref{Gamma1}, \eqref{Ild} and \eqref{Ilddecomp} it follows
\begin{equation}\label{Gamma1decomp}
\Gamma_1(X)=\Gamma_1^{(1)}(X)+\Gamma_1^{(2)}(X)+\Gamma_1^{(3)}(X)\,,
\end{equation}
where
\begin{equation}\label{Gamma1i}
\Gamma_1^{(i)}(X)=\sum\limits_{d\leq D}\chi_4(d)I_{1,d}^{(i)}(X)\,,\;\; i=1,\,2,\,3.
\end{equation}

\subsection{Estimation of $\mathbf{\Gamma_1^{(1)}(X)}$}
\indent

First we consider $\Gamma_1^{(1)}(X)$.
Using formula \eqref{Ild1est} for $J=(X/2,X]$, \eqref{Gamma1i}
and treating the reminder term by the same way as for $\Gamma_3^{(1)}(X)$
we find
\begin{equation} \label{Gamma11est1}
\Gamma_1^{(1)}(X)=\Phi(X)\sum\limits_{d\leq D}\frac{\chi_4(d)}{\varphi(d)}
+\mathcal{O}\bigg(\frac{\varepsilon X^{3-c}}{\log X}\bigg)\,,
\end{equation}
where
\begin{equation*}
\Phi(X)=\int\limits_{-\infty}^{\infty}\Theta(t)I^3(t)e(-Nt)\,dt\,.
\end{equation*}
By  Lemma \ref{IIIest}  we get
\begin{equation}\label{Philowerbound}
\Phi(X)\gg\varepsilon X^{3-c}\,.
\end{equation}
According to (\cite{Dimitrov2}, p. 14 -- 15) we have
\begin{equation}\label{sumchiphi}
\sum\limits_{d\leq D}\frac{\chi_4(d)}{\varphi(d)}=
\frac{\pi}{4}\prod\limits_p \left(1+\frac{\chi_4(p)}{p(p-1)}\right)+\mathcal{O}\Big(X^{-1/20}\Big)\,.
\end{equation}
From \eqref{Gamma11est1} and \eqref{sumchiphi} we obtain
\begin{equation}\label{Gamma11est2}
\Gamma_1^{(1)}(X)=\frac{\pi}{4}\prod\limits_p \left(1+\frac{\chi_4(p)}{p(p-1)}\right) \Phi(X)
+\mathcal{O}\bigg(\frac{\varepsilon X^{3-c}}{\log X}\bigg)+\mathcal{O}\Big(\Phi(X)X^{-1/20}\Big)\,.
\end{equation}
Now \eqref{Philowerbound} and \eqref{Gamma11est2} imply
\begin{equation}\label{Gamma11est3}
\Gamma_1^{(1)}(X)\gg\varepsilon X^{3-c}\,.
\end{equation}

\subsection{Estimation of $\mathbf{\Gamma_1^{(2)}(X)}$}
\indent

Arguing as in the estimation of $\Gamma_3^{(2)}(X)$ we find
\begin{equation} \label{Gamma12est}
\Gamma_1^{(2)}(X)\ll\frac{\varepsilon X^{3-c}}{\log X}\,.
\end{equation}

\subsection{Estimation of $\mathbf{\Gamma_1^{(3)}(X)}$}
\indent

By \eqref{Ild3est} and \eqref{Gamma1i} it follows that
\begin{equation}\label{Gamma13est}
\Gamma_1^{(3)}(X)\ll\sum\limits_{m<D}\frac{1}{d}\ll \log X\,.
\end{equation}

\subsection{Estimation of $\mathbf{\Gamma_1(X)}$}
\indent

Summarizing  \eqref{Gamma1decomp}, \eqref{Gamma11est3}, \eqref{Gamma12est} and \eqref{Gamma13est} we obtain
\begin{equation} \label{Gamma1est}
\Gamma_1(X)\gg\varepsilon X^{3-c}\,.
\end{equation}

\section{Proof of the Theorem}\label{Sectionfinal}
\indent

Taking into account \eqref{varepsilon}, \eqref{GammaGamma0}, \eqref{Gamma0decomp},
\eqref{Gamm3est}, \eqref{Gamma2est2} and \eqref{Gamma1est} we deduce
\begin{equation*}
\Gamma(X)\gg\varepsilon X^{3-c}=\frac{X^{3-c}(\log\log X)^6}{(\log X)^{\theta_0}}\,.
\end{equation*}
The last lower bound yields
\begin{equation}\label{Lowerbound}
\Gamma(X) \rightarrow\infty \quad \mbox{ as } \quad X\rightarrow\infty\,.
\end{equation}
Bearing in mind  \eqref{Gamma} and \eqref{Lowerbound} we establish Theorem \ref{Theorem}.

\vskip20pt
\footnotesize
\begin{flushleft}
S. I. Dimitrov\\
Faculty of Applied Mathematics and Informatics\\
Technical University of Sofia \\
8, St.Kliment Ohridski Blvd. \\
1756 Sofia, BULGARIA\\
e-mail: sdimitrov@tu-sofia.bg\\
\end{flushleft}


\begin{thebibliography}{0}

\bibitem{Baker} R. Baker, {\it Some diophantine equations and inequalities with primes},
Funct. Approx. Comment. Math., \textbf{64} (2), (2021), 203 -- 250.

\bibitem{Baker-Weingartner} R. Baker, A. Weingartner, {\it A ternary diophantine inequality over primes},
Acta Arith., {\bf162}, (2014), 159 -- 196.

\bibitem{Cai1} Y. Cai, {\it On a diophantine inequality involving prime numbers} (in Chinese),
Acta Math Sinica, {\bf39}, (1996), 733 -- 742.

\bibitem{Cai2} Y. Cai, {\it On a diophantine inequality involving prime numbers III},
Acta Mathematica Sinica, English Series, {\bf15}, (1999), 387 -- 394.

\bibitem{Cai3} Y. Cai, {\it A ternary Diophantine inequality involving primes},
Int. J. Number Theory, {\bf14}, (2018), 2257 -- 2268.

\bibitem{Cao-Zhai}  X. Cao, W. Zhai, {\it A Diophantine inequality with prime numbers},
Acta Math. Sinica, Chinese Series, {\bf45}, (2002), 361 -- 370.

\bibitem{Davenport} H. Davenport, {\it Multiplicative number theory} (revised by H. Montgomery),
Third ed., Springer, (2000).

\bibitem{Dimitrov1} S. I. Dimitrov, {\it  A ternary  diophantine inequality over special primes},
JP Journal of Algebra, Number Theory and Applications, \textbf{39}, 3, (2017), 335 -- 368.

\bibitem{Dimitrov2} S. I. Dimitrov, {\it Diophantine approximation with one prime of the form $p=x^2+y^2+1$},
Lith. Math. J., \textbf{61}, 4, (2021), 445 -- 459.

\bibitem{Graham-Kolesnik} S. W. Graham, G. Kolesnik, {\it Van der Corput's Method of Exponential Sums},
Cambridge University Press, New York, (1991).

\bibitem{Heath} D. R. Heath-Brown, {\it The Piatetski-Shapiro prime number theorem},
J. Number Theory, \textbf{16}, (1983), 242 -- 266.

\bibitem{Hooley} C. Hooley, {\it Applications of sieve methods to the theory of numbers},
Cambridge Univ. Press, (1976).

\bibitem{Iwaniec-Kowalski} H. Iwaniec, E. Kowalski, {\it Analytic number theory},
Colloquium Publications, \textbf{53}, Amer. Math. Soc., (2004).

\bibitem{Karat} A. Karatsuba, {\it Principles of the Analytic Number Theory},
Nauka, Moscow, (1983), (in Russian).

\bibitem{Ku-Ne} A. Kumchev, T. Nedeva,
{\it On an equation with prime numbers}, Acta Arith., {\bf 83},
(1998), 117 -- 126.

\bibitem{Kumchev} A. Kumchev, {\it A diophantine inequality involving prime powers},
Acta Arith., {\bf 89}, (1999), 311 -- 330.

\bibitem{Li2022a} J. Li, F. Xue, M. Zhang, {\it A ternary Diophantine inequality with prime numbers of a special form},
Period. Math. Hungar., \textbf{85}, 1, (2022), 14 -- 31.

\bibitem{Linnik} Ju. Linnik, {\it An asymptotic formula in an additive problem of Hardy and Littlewood},
Izv. Akad. Nauk SSSR, Ser.Mat., {\bf24}, (1960), 629 -- 706 (in Russian).

\bibitem{Shapiro} I. Piatetski-Shapiro, {\it On a variant of the Waring-Goldbach problem},
Mat. Sb., {\bf30}, (1952), 105 -- 120, (in Russian).

\bibitem{Sargos-Wu} P. Sargos, J. Wu, {\it Multiple exponential sums with monomials and their applications in number theory},
Acta Math. Hungar., \textbf{87}, (2000), 333 -- 354.

\bibitem{Tenenbaum} G. Tenenbaum, {\it Introduction to Analytic and Probabilistic Number Theory},
Cambridge Univ. Press, (1995).

\bibitem{Titchmarsh} E. Titchmarsh, {\it The Theory of the Riemann Zeta-function}
(revised by D. R. Heath-Brown), Clarendon Press, Oxford (1986).

\bibitem{Tolev1} D. Tolev, {\it On a diophantine inequalityinvolving prime numbers}, 
Acta Arith., {\bf 61}, (1992), 289 -- 306.

\bibitem{Tolev2} D. Tolev, {\it On a diophantine inequality with prime numbers of a special type},
Proc. Steklov Inst. Math., \textbf{299}, (2017), 261 -- 282.

\bibitem{Vaughan} R. C. Vaughan, {\it An elementary method in prime number theory},
Acta Arith., \textbf{37}, (1980), 111 -- 115.

\bibitem{Zhu} L. Zhu, {\it A ternary diophantine inequality with prime numbers of a special type}, 
Proc. Indian Acad. Sci. Math. Sci., \textbf{130}, (2020), Art. 23. 

\end{thebibliography}
\end{document}